\documentclass[11pt,twoside]{amsart}

\usepackage[english]{babel}
\usepackage{csquotes}
\usepackage{amsfonts}
\usepackage{hyperref}
\usepackage{etoolbox}
\usepackage{amsmath}
\usepackage{amssymb}
\usepackage{amsthm}
\usepackage{stmaryrd}
\usepackage{dsfont}
\usepackage{pifont}
\usepackage{mathrsfs}
\usepackage{graphicx}
\usepackage{enumerate}
\usepackage[all]{xy}
\usepackage{enumitem}
\usepackage{geometry}
\usepackage{amsfonts}
\usepackage{fancyhdr}
\usepackage{calc}
\usepackage{cancel}
\usepackage{accents}
\usepackage{abraces}
\usepackage{wrapfig}
\usepackage{tikz,tikz-3dplot}
\usetikzlibrary{decorations.markings}
\usetikzlibrary{shapes,arrows,patterns}

\definecolor{cof}{RGB}{219,144,71}
\definecolor{pur}{RGB}{186,146,162}
\definecolor{greeo}{RGB}{91,173,69}
\definecolor{greet}{RGB}{52,111,72}

\tdplotsetmaincoords{70}{165}

\newcommand{\changefont}{%
    \fontsize{8}{8}\selectfont
}
\title[Cellularization of spherical space forms and the flag manifold of $SL_3(\R)$]{Cellularization for exceptional spherical space forms and the flag manifold of $SL_3(\mathbb{R})$}

\author{Rocco Chiriv\`i, Arthur Garnier and Mauro Spreafico}

\theoremstyle{plain}
\newtheorem{prop}{Proposition}[subsection]
\newtheorem{prop-def}[prop]{Proposition-Definition}

\newtheorem{lem}[prop]{Lemma}
\newtheorem{theo}[prop]{Theorem}
\newtheorem{cor}[prop]{Corollary}
\newtheorem{rem}[prop]{Remark}

\newtheorem*{prop*}{Proposition}
\newtheorem*{prop-def*}{Proposition-Definition}
\newtheorem*{propri*}{Property}
\newtheorem*{lem*}{Lemma}
\newtheorem*{theo*}{Theorem}
\newtheorem*{cor*}{Corollary}
\newtheorem*{rem*}{Remark}
\newtheorem*{definition*}{Definition}
\newtheorem*{exemple*}{Example}
\newtheorem*{notation*}{Notation}

\newcommand{\lra}{\longrightarrow}

\newcommand{\ra}{\rightarrow}
\newcommand{\sdp}{\times\kern-.2em\vrule height1.1ex depth-.05ex}
\newcommand{\epi}{\lra \kern-.8em\ra}
\newcommand{\C}{{\mathbb C}}

\newcommand{\N}{{\mathbb N}}
\newcommand{\R}{{\mathbb R}}

\newcommand{\Z}{{\mathbb Z}}

\newcommand{\ho}{\mathrm{Hom}\,}

\newcommand{\Sph}{\mathbb{S}}

\newcommand{\Pro}{\mathbb{P}}

\newcommand{\Sym}{\mathfrak{S}}
\newcommand{\BT}{\mathcal{T}}
\newcommand{\BO}{\mathcal{O}}
\newcommand{\BI}{\mathcal{I}}
\newcommand{\Pol}{\mathscr{P}}
\newcommand{\Alt}{\mathfrak{A}}
\newcommand{\interior}[1]{\accentset{\circ}{#1}}

\DeclareMathOperator\conv{conv}
\DeclareMathOperator\vertices{vert}

\newcommand{\bigslant}[2]{{\raisebox{.2em}{$#1$}\left/\raisebox{-.2em}{$#2$}\right.}}

\newcommand{\eq}[1][r]
{\ar@<-3pt>@{-}[#1]
\ar@<-1pt>@{}[#1]|<{}="gauche"
\ar@<+0pt>@{}[#1]|-{}="milieu"
\ar@<+1pt>@{}[#1]|>{}="droite"
\ar@/^2pt/@{-}"gauche";"milieu"
\ar@/_2pt/@{-}"milieu";"droite"}

 \normalbaselines
\makeatletter
\newlength\@SizeOfCirc%
\newcommand{\CircleArrowRight}[1]{%
    \setlength{\@SizeOfCirc}{\maxof{\widthof{#1}}{\heightof{#1}}}%
    \tikz [x=1.0ex,y=1.0ex,line width=.12ex]%
        \draw [->,anchor=center]%
            node (0,0) {#1}%
            (0,0.8\@SizeOfCirc) arc (85:-240:0.8\@SizeOfCirc);%
}%
\newcommand{\CircleArrowLeft}[1]{%
    \setlength{\@SizeOfCirc}{\maxof{\widthof{#1}}{\heightof{#1}}}%
    \tikz [x=1.0ex,y=1.0ex,line width=.12ex]%
        \draw [<-,anchor=center]%
            node (0,0) {#1}%
            (0,0.8\@SizeOfCirc) arc (85:-240:0.8\@SizeOfCirc);%
}%
\makeatother


\date{\today}

\begin{document}

\begin{abstract}
We construct an explicit equivariant cellular decomposition of the $(4n-1)$-sphere with respect to binary polyhedral groups, and describe the associated cellular homology chain complex. 
\newline
\indent As a corollary of the binary octahedral case, we deduce an $\Sym_3$-equivariant decomposition of the flag manifold of $SL_3(\R)$.
\end{abstract}

\maketitle


\section{Introduction}


Given a finite group acting freely on a compact topological manifold, it is natural to look for an equivariant cellular decomposition: in particular, this provides a bounded cochain complex of free modules over the group, lifting the action on cohomology to the derived category level.

Milnor classified in \cite{milnor_Sn} finite groups acting freely on $\Sph^3$: quaternionic, metacyclic, generalized tetrahedral, binary octahedral and binary icosahedral groups. For all those except the last two, an equivariant decomposition is known \cite{mms13,fgns_meta,fgnsT,chirivi-spreafico}. In the present article, we deal with the two exceptional cases, which we will denote by $\BO$ and $\BI$. We also treat the first tetrahedral group $\BT$, since our technique gives a different construction than that of \cite{fgnsT}.


Note that $\Sph^3/\BI$ is the Poincaré homology sphere (see Theorem \ref{SPHIcells} and Remark \ref{poincare}). The other case is also interesting: as a corollary, we obtain an $\Sym_3$-equivariant decomposition of the flag manifold of $SL_3(\R)$.

More precisely, the flag manifold of $SL_3(\R)$ is $\mathcal{F}(\R)=\{(0\lneq V_1\lneq V_2\lneq \R^3)\}$ and its Weyl group $W=\Sym_3$ acts freely on it. We have a tower of covering maps
\[\xymatrix@R-1pc{\Sph^3 \ar@/_10pt/@{->>}_{/\mathcal{Q}_8}[rrd] \ar@{->>}^<<<<<{/\{\pm1\}}[r] & SO(3) \ar@{->>}^<<<<<{/\{\pm1\}^2}[r] & SO(3)/\{\pm1\}^2 \eq[d] \ar@{->>}^<<<<<{/\Sym_3}[r] & \Sph^3/\BO \\ & & \mathcal{F}(\R) \ar@/_10pt/@{->>}_{/\Sym_3}[ru] & }\]


This will provide an $\Sym_3$-equivariant cellular decomposition of $\mathcal{F}(\R)$, once a $\BO$-equivariant cellular structure on $\Sph^3$ is known. This is actually the motivation of this paper.
\newline
\indent Our strategy is based on the ideas of \cite{chirivi-spreafico}. Given a finite group $G$ acting freely on $\Sph^n$, one looks at the \emph{orbit polytope} $\Pol$, i.e. the convex hull of the orbit of a point in $\Sph^n$. Then, the group acts freely on the \emph{boundary} $\partial\Pol$ of $\Pol$ (see the Theorem \ref{thm6.4}) and an equivariant cellular decomposition is found by decomposing a \emph{fundamental domain} $\mathscr{D}$ for the action on $\partial\Pol$. Next, using the $G$-equivariant homeomorphism $\partial\mathscr{P}\to\mathbb{S}^n$, we obtain a decomposition of the sphere. Furthermore, the fact that the decomposition comes from a decomposition of some polytopal complex with open faces as cells, implies that the resulting decomposition of $\Sph^n$ is regular (that is, the closure of a cell is homeomorphic to a closed ball) and the boundaries are then easily computed. 

The main tool is the theorem \ref{thm6.4}, which essentially says that we can find representatives for the action of $G$ on the \emph{facets} of $\Pol$ such that their union is a fundamental domain. We proceed by determining such representatives for $\BO$ and $\BI$ and $\BT$.
\newline
\indent The main results of this paper may be summarized as follows, combining Theorems \ref{SPHNOcells} and \ref{SPHNIcells}.
\begin{theo*}
Every sphere $\Sph^{4n-1}$, endowed with the natural free and isometric action of $\BO$ (resp. of $\BI$, $\BT$), admits an explicit equivariant cell decomposition. As a consequence, the associated cellular homology chain complex is explicitly given in terms of matrices with entries in the group algebras $\Z[\BO]$, $\Z[\BI]$ and $\Z[\BT]$, respectively.
\end{theo*}
The crucial case is $\Sph^3$. Then one may prove the result inductively, using curved joins. As a consequence, one obtains the following result, which combines the results \ref{resolforO}, \ref{SPHNIcells}, \ref{HstarO} and \ref{HstarI}.
\begin{cor*}
One may give an explicit free $4$-periodic resolution of the trivial module $\Z$ over $\Z[\BO]$, $\Z[\BI]$ and $\Z[\BT]$. In particular, one can compute the cohomology modules $H^*(\BO,M)$, $H^*(\BI,M)$ and $H^*(\BT,M)$ for any $\Z[G]$-module $M$.
\end{cor*}
\indent It should be noted that such resolutions were already given in \cite{tomoda-zvengrowski_2008}. Our approach however has the advantage of being more conceptual and geometric. Moreover, using the first result above we can derive the following consequence (see Theorems \ref{S3equiv}, \ref{S3actioncohomology} and Corollary \ref{S3actioncohomology}):
\begin{theo*}
The flag manifold $\mathcal{F}(\R)$ admits an explicit equivariant cell decomposition, with respect to its Weyl group $\Sym_3$. In particular, its cellular homology chain complex is explicitly given in terms of matrices with entries in $\Z[\Sym_3]$ and the isomorphism type of the $\Z[\Sym_3]$-module $H^{\ast}(\mathcal{F}(\R),\Z)$ is determined.
\end{theo*}
\indent Let us outline the content of the article. In Section \ref{orbit}, after a quick reminder on polytopes, we introduce orbit polytopes and study some of their properties. Most importantly, we explain how to obtain a polytopal fundamental domain for the boundary of an orbit polytope, and hence for the sphere, using the radial projection. Most of those results appeared in \cite{chirivi-spreafico}, we recall them for the convenience of the reader. 
\newline
\indent In Section \ref{binary}, we introduce the binary polyhedral groups as finite subgroups of unit quaternions and \emph{spherical space forms}.
\newline
\indent In Sections \ref{octahedral}, \ref{icosahedral} and \ref{tetrahedral}, we apply the orbit polytope techniques to the cases where $G$ is $\BO$, $\BI$ or the binary tetrahedral group $\BT$ acting on $\Sph^3$. In particular, we explicitly describe a fundamental domain for the boundary of the polytope, and we use it to determine a $G$-equivariant cellular decomposition of $\Sph^3$. Moreover, we compute the resulting cellular homology chain complexes (which are bounded complexes of free $\Z[G]$-modules). Finally, we generalize this to $\Sph^{4n-1}$ and use the resulting equivariant cellular decomposition to obtain an explicit 4-periodic free resolution of $\Z$ over $\Z[G]$ and recover the integral cohomology of $G$. Moreover, in Section \ref{octahedral}, the application to the real flag manifold of $SL_3(\R)$ is derived.

\section{Orbit polytopes}\label{orbit}
\indent The following section gives the main tools for determining fundamental domains for finite groups acting isometrically on the sphere $\Sph^3$, by using their orbit polytopes. We recall results from \cite{chirivi-spreafico}. For general properties of polytopes, the reader is referred to \cite{ziegler-poly}.
\subsection{Some general facts on polytopes}
\hfill\\

We denote by $\Sph^{n-1}:=\{x\in\R^n~;~|x|=1\}$ the $(n-1)$-dimensional sphere and by $\mathbb{D}^n:=\{x\in\R^n~;~|x|\le 1\}$ the $n$-dimensional disc. To a set of points $X$ in $\R^n$, one can associate its \emph{convex hull} denoted by $\mathrm{conv}(X)$.
\newline
\indent The convex hull $\Pol=\mathrm{conv}(x_1,\dotsc,x_n)$ of a finite set of points is called a \emph{polytope}. The \emph{dimension} $\dim(\Pol)$ of $\Pol$ is the dimension of the affine subspace generated by the $x_i$'s. A polytope can also be defined as a bounded set given by the intersection of a finite number of closed half-spaces (see \cite{ziegler-poly}).
\newline
\indent A face of $\Pol$ is the intersection of $\Pol$ with an affine hyperplane $\mathcal{H}$ such that $\Pol$ is entirely contained in one of the closed half-spaces defined by $\mathcal{H}$. A \emph{proper face} of $\Pol$ is a face $F$ such that $F\ne\Pol$. The \emph{dimension} of a face $F$ is the dimension of the affine space it generates. The faces of $\Pol$ of dimension 0, 1 or $\dim\Pol-1$ are called \emph{vertices}, \emph{edges} and \emph{facets}, respectively. The \emph{boundary} $\partial\Pol$ of $\Pol$ is the union of all the faces of $\Pol$ of dimension smaller than $\dim\Pol$. A point of $\Pol$ is said to be an \emph{interior point} if it doesn't belong to $\partial\Pol$. The set of $d$-faces of $\Pol$ (i.e. of $d$-dimensional faces of $\Pol$) is denoted by $\Pol_d$. Usually, we denote also $\vertices(\Pol):=\Pol_0$. When we want to stress the vertices of $F$, we write $F=[v_1,\dotsc,v_k]$ if $\{v_1,\dotsc,v_k\}=\vertices(F)=F\cap\vertices(\Pol)$.

\subsection{Finite group acting freely on $\Sph^n$, orbit polytope and fundamental domains}
\hfill\\

Let $G\subset O(n)$ be a finite group acting freely on a sphere $\Sph^{n-1}\subset\R^n$ and such that any of its orbits span $\R^n$. Fix $v_0\in\Sph^{n-1}$ and let $\Pol:=\conv(G\cdot v_0)$ be the associated \emph{orbit polytope}.
\newline
\indent Recall that, if a group $G$ acts on a topological space $X$, then a \emph{fundamental domain} for the action of $G$ on $X$ is a subset $\mathcal{D}$ of $X$ such that, for $g\ne h\in G$, the set $g\mathcal{D}\cap h\mathcal{D}$ has empty interior and the translates of $\mathcal{D}$ cover $X$, i.e. $X=\bigcup_{g\in G}g\mathcal{D}$.
\begin{theo}\emph{(\cite[6.1-6.4]{chirivi-spreafico})}\label{thm6.4}
\begin{enumerate}[label=\roman*)]
\item If $F$ and $F'$ are distinct proper faces of $\Pol$ of the same dimension, then $F\cap gF'$ has empty interior for every $1\ne g\in G$.
\item The group $G$ acts freely on the set $\Pol_d$ of $d$-dimensional faces of $\Pol$, for every $0\le d<\dim(\Pol)$. 
\item Moreover, the origin $0$ is an interior point of $\Pol$ and we have a $G$-equivariant homeomorphism
\[\begin{array}{ccc}
\partial\Pol & \stackrel{\tiny{\sim}}\to  & \Sph^{n-1} \\ x & \mapsto & {x}/{|x|}\end{array}\]
\item Given a system of representatives $F_1,\dotsc,F_r$ for the (free) action of $G$ on the set of facets of $\Pol_G$ such that the union $\bigcup_i F_i$ is connected, then this union is a fundamental domain for the action of $G$ on $\partial\Pol_G$. Furthermore, there exists such a system.
\end{enumerate}
\end{theo}
\indent We finish this section by giving a simple but useful fact.
\begin{prop}\label{VcapVm}
Given distinct facets $F_1,\dotsc,F_r$ of $\Pol$, form their union $\mathcal{D}:=\bigcup_{i=1}^r F_i$, consider the subset $V$ of $G$ defined by $\vertices(\mathcal{D})=V\cdot v_0$ and assume that $v_0\in\bigcap_{i=1}^r \vertices(F_i)$. If $V\cap V^{-1}=\{1\}$, then the $F_i$'s belong to distinct $G$-orbits.
\newline
\indent If $r|G|=|\Pol_{n-1}|$, then $\mathcal{D}$ is a fundamental domain for the action of $G$ on $\partial\Pol$.
\end{prop}
\begin{proof}
Suppose that there are $1\le i\ne j\le r$ and $g\in G$ such that $F_j=gF_i$. Since $v_0\in\vertices(F_i)$, we get $gv_0\in g\vertices(F_i)=\vertices(gF_i)=\vertices(F_j)$, so $g\in V$. On the other hand, $v_0\in\vertices(F_j)=g\vertices(F_i)$, hence $g^{-1}v_0\in\vertices(F_i)$, that is $g^{-1}\in V$. Therefore $g\in V\cap V^{-1}$, so $g=1$ and thus $F_i=F_j$, a contradiction.
\newline
\indent Now, the equation $r|G|=|\Pol_{n-1}|$ ensures that $F_1,\dotsc,F_r$ is a system of representatives of facets and the condition $v_0\in\bigcap_i\vertices(F_i)$ shows that $\mathcal{D}$ is connected, hence the second statement follows from the theorem \ref{thm6.4}. 
\end{proof}

\subsection{The curved join}
\hfill\\

Here, we shall define the notion of \emph{curved join}, which allows one to describe the fundamental domain for $\partial\Pol_G$ as a subset of the sphere. It will also be used to reduce the higher dimensional cases $\Sph^{4n-1}$ to $\Sph^3$. For any detail, see \cite[\S 2.4]{fgns_meta}.

Given $W_1,W_2\subset\Sph^{n-1}\subset\R^n$ such that $W_1\cap (-W_2)=\emptyset$, we define their \emph{curved join} $W_1\ast W_2$ as the projection on $\Sph^{n-1}$ of $\mathrm{conv}(W_1\cup W_2)$. For instance we have
\[\Sph^1\ast \Sph^1=\Sph^3.\]
This generalizes as follows: identifying $\C^m$ with $\R^{2m}$ and given the standard orthonormal basis $\{e_1,\dotsc,e_{2m}\}$ of $\R^{2m}$, for each $2\le r\le 2m$, denote by $\Pi_r$ the plane generated by $\{e_{r-1},e_r\}$. Suppose $\Pi_{r_1}\cap\Pi_{r_2}=0$ and let $W_1$ and $W_2$ be subsets of the unit circles of $\Pi_{r_1}$ and $\Pi_{r_2}$, respectively. Then, one can define the curved join $W_1\ast W_2$ as above. In particular, we denote by $\Sigma_k$ the unit circle lying in the $k^{\text{th}}$ copy of $\C$ in $\C^m$ and we have the following equality
\[\Sph^{2m-1}=\Sigma_1\ast\Sigma_2\ast\cdots\ast\Sigma_m.\]

Let $G\le O(n)$ be a finite group acting freely on $\Sph^{n-1}$ and let $h\in\N^*$. Then, we can make $G$ act diagonally on $\R^{hn}$.
Under the identification $\Sph^{hn-1}=\Sph^{(h-1)n-1}\ast\Sph^{n-1}$, we have $g\cdot(x\ast y)=gx\ast gy$.
\newline
\indent To compute the boundaries, we shall need the following technical result:
\begin{lem}\emph{(\cite[Lemma 2.5]{fgns_meta})}\label{bound_join}
We have the following Leibniz formula for the oriented boundary of a curved join
\[\partial(X\ast Y)=\partial X\ast Y-(-1)^{\dim X}X\ast\partial Y.\]
\end{lem}
\indent In fact, we will use the following general lemma, allowing to recursively determine a fundamental domain and an equivariant cellular decomposition on $\Sph^{hn-1}$, once we know one on $\Sph^{n-1}$.
\newline
\indent More precisely, let $G\le O(n)$ be a finite group acting freely on $\Sph^{n-1}$. Assume that $\mathcal{D}$ is a fundamental domain for the action on $\Sph^{n-1}$ and that $\widetilde{L}$ is a cellular decomposition of $\mathcal{D}$. We obtain an equivariant cell decomposition $\widetilde{K}=G\cdot\widetilde{L}$ of $\Sph^{n-1}$ and $L=\widetilde{K}/G$ is a cellular decomposition of $\Sph^{n-1}/G$. Assume further that $\widetilde{Z}$ is a subcomplex of $\widetilde{L}$ that is a minimal decomposition of $\mathcal{D}$ by lifts of the cells of $L$.
\newline
\indent Let $h\in\N^*$ and consider the diagonal action of $G$ on $\Sph^{hn-1}$. Then, a fundamental domain for this action on $\Sph^{hn-1}$ is given by
\[\mathcal{D}':=\Sph^{(h-1)n-1}\ast\mathcal{D}.\]
\indent Furthermore, we construct an equivariant cellular decomposition $\widetilde{K'}$ of $\Sph^{hn-1}$ and a minimal cellular decomposition $\widetilde{L'}$ of $\mathcal{D}'$ as follows:
\begin{enumerate}[label=$\bullet$]
\item the $(h-1)n-1$-skeleton of $\widetilde{L'}$ is $\widetilde{L'}_{(h-1)n-1}=\widetilde{K}$;
\item for the $(h-1)n$-skeleton of $\widetilde{L'}$, we attach $k_0(h-1)n$-cells to $\widetilde{K}$, where $k_0$ is the number of $0$-cells $\widetilde{e}^0_l$ of $\widetilde{Z}$ and the corresponding attaching map is given by the parametrization of the curved join $\widetilde{K}\ast\widetilde{e}^0_l$;
\item for the $(h-1)n+1$-skeleton of $\widetilde{L'}$, we attach $k_1(h-1)n+1$-cells to the $(h-1)n$-skeleton of $\widetilde{L'}$, where $k_1$ is the number of $1$-cells $\widetilde{e}^1_l$ of $\widetilde{Z}$ and the attaching map is given by the parametrization of $\widetilde{L'}_{(h-1)n}\ast\widetilde{e}^1_l$;
\item we carry on this procedure up to dimension $hn-1$.
\end{enumerate}
\indent We can summarize this in the following result.
\begin{lem}\emph{(\cite[Lemma 4.1]{fgns_meta})}\label{lem_join}
\newline
If $G\le O(n)$ is a finite group acting freely on $\Sph^{n-1}$, if $\mathcal{D}$ is a fundamental domain for this action and if $\widetilde{L}$ is a cellular decomposition of $\mathcal{D}$, with associated equivariant cellular decomposition $\widetilde{K}=G\cdot\widetilde{L}$ of $\Sph^{n-1}$, then for every $h\in\N^*$, the subset
\[\mathcal{D}':=\Sph^{(h-1)n-1}\ast\mathcal{D}\]
is a fundamental domain for the diagonal action of $G$ on $\Sph^{hn-1}$ and the above construction gives a cell decomposition $\widetilde{L'}$ of $\mathcal{D}'$, with associated equivariant cell decomposition $\widetilde{K'}:=G\cdot\widetilde{L'}$ of $\Sph^{hn-1}$.
\end{lem}

\section{Binary spherical space forms}\label{binary}
\subsection{Binary polyhedral groups}
\hfill\\

Consider the \emph{quaternion group} $\mathcal{Q}_8:=\left<i,j\right>=\{\pm1,\pm i,\pm j,\pm k\}$, a finite subgroup of the sphere $\Sph^3$ of unit quaternions. The element $\varpi:=\tfrac 1 2 (-1+i+j+k)$ has order $3$ and normalizes $\mathcal{Q}_8$. Hence, the group 
\[\BT:=\left<i,\varpi\right>\]
has order 24, and the 16 elements of $\BT\setminus\mathcal{Q}_8$ have the form $\frac{1}{2}(\pm1\pm i\pm j\pm k)$. The group $\BT$ is the \emph{binary tetrahedral group}.

Next, the element $\gamma:=\tfrac {1}{\sqrt{2}}(1+i)$ has order $8$ and normalizes both $\mathcal{Q}_8$ and $\BT$. Hence the group 
\[\BO:=\left<\varpi,\gamma\right>\]
is of order 48 (since $\gamma^2=i$) and is called the \emph{binary octahedral group} and we have $\BO=\left<\varpi,\gamma\right>$. The set $\BO\setminus\BT$ consists of the 24 elements $\frac{1}{\sqrt{2}}(\pm u\pm v)$ where $u\ne v\in\{1,i,j,k\}$. 

Setting $\varphi:=\tfrac{1}{2}{(1+\sqrt{5})}$, the element $\sigma:=\tfrac{1}{2}(\varphi^{-1}+i+\varphi j)$ is of order $5$ hence the \emph{binary icosahedral group}
\[\BI:=\left<i,\sigma\right>\]
has order 120 and we have $\BT\le\BI$.

The universal covering map $\Sph^3 = SU(2) \twoheadrightarrow SO(3)$ can be interpreted as the action of unit quaternions on the space of purely imaginary quaternions
\[\mathrm{B} : \Sph^3 \to SO_3(\R).\]
The respective images of $\BT$, $\BO$ and $\BI$ are the rotation groups $\Alt_4$, $\Sym_4$ and $\Alt_5$ of a regular tetrahedron, octahedron and icosahedron respectively, hence the names.

It has been observed by Coxeter and Moser in \cite[\S 6.4]{coxeter-moser} that finite subgroups of $\Sph^3$ have nice presentation. Namely, denoting
\[\left<\ell,m,n\right>:=\left<r,s,t~|~r^\ell=s^m=t^n=rst\right>,\]
we have isomorphisms
\[\left<2,3,3\right>\simeq \BT,~~\left<2,3,4\right>\simeq \BO,~~\left<2,3,5\right>\simeq \BI.\]

Finally, for $n\in\N^*$ and $\mathcal{G}\in\{\BT,\BO,\BI\}$, we define the \emph{polyhedral spherical space form}
\[\mathsf{P}_\mathcal{G}^{4n-1}:=\Sph^{4n-1}/\mathcal{G}.\]

\section{The octahedral case}\label{octahedral}

In the following two sections, we let both $\BO$ and $\BI$ act (freely) by (quaternion) multiplication on the left on $\Sph^3$.
\subsection{Fundamental domain}\label{3.1}
\hfill\\

We use Theorem \ref{thm6.4} to find a fundamental domain for $\BO$ on $\Sph^3$. To this end, we first introduce the \emph{orbit polytope} in $\R^4$
\[\Pol:=\conv(\BO).\]
Then, we know that $\BO$ acts freely on the set $\Pol_3$ of facets of $\Pol$ and by Theorem \ref{thm6.4}, it suffices to find a set of representatives in $\Pol_3$ such that their union is connected; this will be a fundamental domain for the action on $\partial\Pol$, which we can transport to the sphere $\Sph^3$ using the equivariant homeomorphism $\partial\Pol \to \Sph^3$, $x \mapsto x/|x|$. 
\newline
\indent The 4-polytope $\Pol$ has 48 vertices, 336 edges, 576 faces and 288 facets and is known as the \emph{disphenoidal 288-cell}; it is dual to the bitruncated cube. Since $\BO$ acts freely on $\Pol_3$, there must be exactly six orbits in $\Pol_3$. We introduce the following elements of $\BO$, also expressed in terms of the generators $s$ and $t$ in the Coxeter-Moser presentation: 
\[\left\{\begin{array}{ll}
\omega_0:=\frac{1+i+j+k}{2}=s, \\[.5em]
\omega_i:=\frac{1-i+j+k}{2}=t^{-1}st^{-1}, \\[.5em]
\omega_j:=\frac{1+i-j+k}{2}=s^{-1}t^2, \\[.5em]
\omega_k:=\frac{1+i+j-k}{2}=t^{-1}st.\end{array}\right.~~~~\text{and}~~~~\left\{\begin{array}{ll}
\tau_i:=\frac{1+i}{\sqrt{2}}=t, \\[.5em]
\tau_j:=\frac{1+j}{\sqrt{2}}=t^{-1}s, \\[.5em]
\tau_k:=\frac{1+k}{\sqrt{2}}=st^{-1}.\end{array}\right.\]
\newline
\indent Next, we may find explicit representatives for the $\BO$-orbits of $\Pol_3$.
\begin{prop}\label{Ofunddom}
The following tetrahedra (in $\R^4$)
\[\Delta_1:=[1,\tau_i,\tau_j,\omega_0],~\Delta_2:=[1,\tau_j,\tau_k,\omega_0],~\Delta_3:=[1,\tau_k,\tau_i,\omega_0],\]
\[\Delta_4:=[1,\tau_i,\omega_k,\tau_j],~\Delta_5:=[1,\tau_j,\omega_i,\tau_k],~\Delta_6:=[1,\tau_i,\omega_j,\tau_k]\]
form a system of representatives of $\BO$-orbits of facets of $\Pol$. Furthermore, the subset of $\Pol$ defined by
\[\mathscr{D}:=\bigcup_{i=1}^6 \Delta_i\]
is a (connected) polytopal complex and is a fundamental domain for the action of $\BO$ on $\partial\Pol$.
\end{prop}
\begin{proof}
First, we have to find the facets of $\Pol$ by giving the defining inequalities. To do this, we make the group $\{\pm1\}^4\rtimes\Sym_4$ act on $\R^4$ by signed permutations of coordinates. Let
\[v_1:=\left(\begin{smallmatrix}3-2\sqrt{2} \\ \sqrt{2}-1 \\ \sqrt{2}-1 \\ 1\end{smallmatrix}\right),~~v_2:=\left(\begin{smallmatrix}2-\sqrt{2} \\ 2-\sqrt{2} \\ 2\sqrt{2}-2 \\ 0\end{smallmatrix}\right).\]
By invariance of $\Pol$, to prove that the 288 inequalities $\left<v,x\right>\le1$, with $v\in (\{\pm1\}^4\rtimes\Alt_4)\cdot\{v_1,v_2\}$, are valid for $\Pol$, it suffices to check the two inequalities $\left<v_i,x\right>\le1$, for $i=1,2$.
As there are indeed 288 conditions, we have in fact all of them, hence the facets are given by the equalities $\left<v,x\right>=1$ and we find their vertices by looking at vertices of $\Pol$ that satisfy these equalities. We find
\[\vertices(\mathscr{D})=\{1,\tau_i,\tau_j,\tau_k,\omega_i,\omega_j,\omega_k,\omega_0\}.\]
Now, since $\R^4=\mathrm{span}(\BO)$ and $\vertices(\mathscr{D})\cap\vertices(\mathscr{D})^{-1}=\{1\}$, Proposition \ref{VcapVm} ensures that $\mathscr{D}$ is indeed a fundamental domain for $\partial\Pol$.
\end{proof}
\begin{rem}\label{methO}
The recipe used to find these tetrahedra is quite simple. First, choose $\Delta_1$ in some $\BO$-orbit of $\partial\Pol_3$ and containing $1$ as a vertex. Then, we arbitrarily choose another orbit and look at the dimensions of the intersections of $\Delta_1$ with the facets of this second orbit. There is exactly one facet (namely $\Delta_2$) for which the intersection has dimension 2 and we continue further until we obtain representatives for the six orbits. Hence, a lot of different fundamental domains can be produced in this way. The calculations can be done using the Maple package ``Convex'' (see \cite{convex}) and quaternionic multiplication, as in \cite{GAP4}.
\end{rem}

It should be noted that all the figures displayed in the sequel only reflect the combinatorics of the polytopes we consider, not the metric they carry as subsets of $\Sph^3$.

\begin{center}
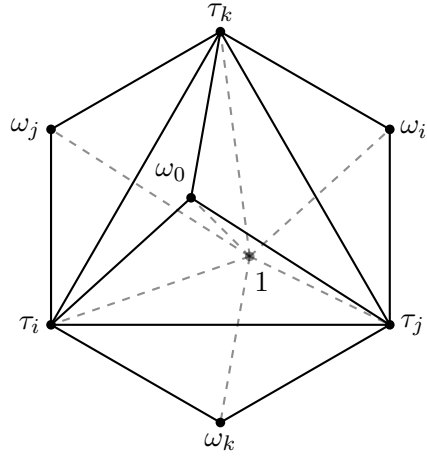
\begin{figure}[h!]
\begin{tikzpicture}[scale=2.6,rotate=90,thick]
  \coordinate (ti) at (-0.5,0.866);
  \coordinate (tk) at (1,0);
  \coordinate (ok) at (-1,0);
  \coordinate (oi) at (0.5,-0.866);
  \coordinate (oj) at (0.5,0.866);
  \coordinate (tj) at (-0.5,-0.866);
  \coordinate (u) at (-0.15,-0.15);
  \coordinate (o0) at (0.15,0.15);
  
  \draw (-0.28,-0.12) node[right]{$1$};
  \draw (oi) node[right]{$\omega_i$};
  \draw (oj) node[left]{$\omega_j$};
  \draw (ok) node[below]{$\omega_k$};
  \draw (ti) node[left]{$\tau_i$};
  \draw (tj) node[right]{$\tau_j$};
  \draw (tk) node[above]{$\tau_k$};
  \draw (0.28,0.12) node[left]{$\omega_0$};
  
  \draw (oj)--(tk);
  \draw (tk)--(oi);
  \draw (oi)--(tj);
  \draw (tj)--(ok);
  \draw (ok)--(ti);
  \draw (ti)--(oj);
  \draw (ti)--(o0);
  \draw (tk)--(o0);
  \draw (tj)--(o0);
  \draw[dashed,opacity=0.45] (u)--(ti);
  \draw[dashed,opacity=0.45] (u)--(tj);
  \draw[dashed,opacity=0.45] (u)--(tk);
  \draw[dashed,opacity=0.45] (u)--(oi);
  \draw[dashed,opacity=0.45] (u)--(oj);
  \draw[dashed,opacity=0.45] (u)--(ok);
  \draw[dashed,opacity=0.45] (u)--(o0);
  \draw (ti)--(tj);
  \draw (tj)--(tk);
  \draw (tk)--(ti);
  
  \fill[fill=black] (ti) circle (0.7pt);
  \fill[fill=black] (tj) circle (0.7pt);
  \fill[fill=black] (tk) circle (0.7pt);
  \fill[fill=black] (oi) circle (0.7pt);
  \fill[fill=black] (oj) circle (0.7pt);
  \fill[fill=black] (ok) circle (0.7pt);
  \fill[fill=black,opacity=0.45] (u) circle (0.7pt);
  \fill[fill=black] (o0) circle (0.7pt);
\end{tikzpicture}
\caption{The six tetrahedra inside $\mathscr{D}$.}
\end{figure}
\end{center}

\subsection{Associated $\BO$-equivariant cellular decomposition of $\partial\Pol$}\label{3.2}
\hfill\\

We shall now examine the combinatorics of the polytopal complex $\mathscr{D}$ constructed in the previous subsection to obtain a cellular decomposition of it. Since $\mathscr{D}$ is a fundamental domain for $\BO$ on $\partial\Pol$, translating the cells will give an equivariant decomposition of $\partial\Pol$ and projecting to $\Sph^3$ will give the desired equivariant cellular structure on the sphere.
\newline
\indent The facets of $\mathscr{D}$ are the ones of the six tetrahedra $\Delta_i$, except those that are contained in some intersection $\Delta_i\cap\Delta_j$. We obtain the following facets
\[\mathscr{D}_2=\{[1,\tau_j,\omega_i],[1,\omega_i,\tau_k],[1,\tau_k,\omega_j],[1,\omega_j,\tau_i],[1,\tau_i,\omega_k],[1,\omega_k,\tau_j],[\tau_j,\omega_i,\tau_k],\]
\[[\tau_k,\omega_j,\tau_i],[\tau_i,\omega_k,\tau_j],[\tau_i,\tau_j,\omega_0],[\tau_j,\tau_k,\omega_0],[\tau_k,\tau_i,\omega_0]\}.\]
\newline
\indent We notice the following relations
\[\left\{\begin{array}{ll}
\tau_i\cdot [1,\tau_j,\omega_i]=[\tau_i,\omega_0,\tau_k], \\
\tau_i\cdot [1,\omega_i,\tau_k]=[\tau_i,\tau_k,\omega_j], \end{array}\right.~~\left\{\begin{array}{ll}
\tau_j\cdot [1,\tau_i,\omega_j]=[\tau_j,\omega_k,\tau_i], \\
\tau_j\cdot [1,\omega_j,\tau_k]=[\tau_j,\tau_i,\omega_0], \end{array}\right.~~\left\{\begin{array}{ll}
\tau_k\cdot [1,\tau_j,\omega_k]=[\tau_k,\omega_i,\tau_j], \\
\tau_k\cdot [1,\omega_k,\tau_i]=[\tau_k,\tau_j,\omega_0]. \end{array}\right.\]
\newline
\indent These are the only relations linking facets, hence, we may gather facets two by two and define the following 2-cells and 1-cells, respectively
\[e^2_1:=]\tau_j,1,\omega_i[~\cup~]1,\omega_i[~\cup~]1,\omega_i,\tau_k[,~e^2_2:=]\tau_i,1,\omega_j[~\cup~]1,\omega_j[~\cup~]1,\omega_j,\tau_k[,~e^2_3:=]\tau_i,1,\omega_k[~\cup~]1,\omega_k[~\cup~]1,\omega_k,\tau_j[,\]
\[e^1_1:=]1,\tau_i[,~e^1_2:=]1,\tau_j[,~e^1_3:=]1,\tau_k[,\]
recalling that, for a polytope $[v_1,\dotsc,v_n]:=\conv(v_1,\dotsc,v_n)$, we denote by $]v_1,\dotsc,v_n[$ its \emph{interior}, namely its maximal face.
\newline
\indent If we add vertices of $\mathscr{D}$ and its interior, which is formed by only one cell $e^3$ by construction, then we may cover all of $\mathscr{D}$ with these cells and some of their translates. Thus, we have obtained the
\begin{lem}\label{ODcells}
Consider the following sets of cells in $\mathscr{D}$
\[\left\{\begin{array}{ll}
E^0_\mathscr{D}:=\{1,~\tau_i,~\tau_j,~\tau_k,~\omega_i,~\omega_j,~\omega_k\}, \\[.5em] 
E^1_\mathscr{D}:=\{e^1_1,~\tau_j e^1_1,~\tau_k e^1_1,~\omega_i e^1_1,~e^1_2,~\tau_i e^1_2,~\tau_k e^1_2,~\omega_j e^1_2,~e^1_3,~\tau_i e^1_3,~\tau_j e^1_3,~\omega_k e^1_3\}, \\[.5em] 
E^2_\mathscr{D}:=\{e^2_1,~\tau_i e^2_1,~e^2_2,~\tau_j e^2_2,~e^2_3,~\tau_k e^2_3\}, \\[.5em] 
E^3_\mathscr{D}:=\{e^3\} \end{array}\right.\]
Then, one has the following cellular decomposition of the fundamental domain 
\[\mathscr{D}=\coprod_{\substack{0\le j \le 3 \\ e \in E^j_\mathscr{D}}}e.\]
\end{lem}
\begin{center}
\begin{figure}[h!]
\begin{tikzpicture}[scale=2.6,rotate=90,thick]
  \coordinate (ti) at (-0.5,0.866);
  \coordinate (tk) at (1,0);
  \coordinate (ok) at (-1,0);
  \coordinate (oi) at (0.5,-0.866);
  \coordinate (oj) at (0.5,0.866);
  \coordinate (tj) at (-0.5,-0.866);
  \coordinate (u) at (-0.15,-0.15);
  \coordinate (o0) at (0.15,0.15);
  
  \draw (u) node[below]{$1$};
  \draw (oi) node[right]{$\omega_i$};
  \draw (oj) node[left]{$\omega_j$};
  \draw (ok) node[below]{$\omega_k$};
  \draw (ti) node[left]{$\tau_i$};
  \draw (tj) node[right]{$\tau_j$};
  \draw (tk) node[above]{$\tau_k$};
  \draw (o0) node[left]{$\omega_0$};
  
  \draw[decoration={markings, mark=at position 0.5 with {\arrow{>}}},postaction={decorate},color=green] (oj)--(tk);
  \draw[decoration={markings, mark=at position 0.5 with {\arrow{>}}},postaction={decorate},color=green] (tk)--(oi);
  \draw[decoration={markings, mark=at position 0.5 with {\arrow{>}}},postaction={decorate},color=red] (oi)--(tj);
  \draw[decoration={markings, mark=at position 0.5 with {\arrow{>}}},postaction={decorate},color=red] (tj)--(ok);
  \draw[decoration={markings, mark=at position 0.5 with {\arrow{>}}},postaction={decorate},color=blue] (ok)--(ti);
  \draw[decoration={markings, mark=at position 0.5 with {\arrow{>}}},postaction={decorate},color=blue] (ti)--(oj);
  \draw[decoration={markings, mark=at position 0.5 with {\arrow{>}}},postaction={decorate},color=green] (ti)--(o0);
  \draw[decoration={markings, mark=at position 0.5 with {\arrow{>}}},postaction={decorate},color=red] (tk)--(o0);
  \draw[decoration={markings, mark=at position 0.5 with {\arrow{>}}},postaction={decorate},color=blue] (tj)--(o0);
  \draw[decoration={markings, mark=at position 0.5 with {\arrow{>}}},postaction={decorate},color=red,dashed] (u)--(ti);
  \draw[decoration={markings, mark=at position 0.5 with {\arrow{>}}},postaction={decorate},color=green,dashed] (u)--(tj);
  \draw[decoration={markings, mark=at position 0.5 with {\arrow{>}}},postaction={decorate},color=blue,dashed] (u)--(tk);
  
  \draw[rotate=30] (1.1,0) node{$\omega_j e^1_2$};
  \draw[rotate=90] (1.1,0) node{$\tau_i e^1_3$};
  \draw[rotate=150] (1.1,0) node{$\omega_k e^1_3$};
  \draw[rotate=210] (1.1,0) node{$\tau_j e^1_1$};
  \draw[rotate=270] (1.1,0) node{$\omega_i e^1_1$};
  \draw[rotate=330] (1.1,0) node{$\tau_k e^1_2$};
  \draw (0.6,0.3) node{$\tau_k e^1_1$};
  \draw (0.1,0.65) node{$\tau_i e^1_2$};
  \draw (0.1,-0.4) node{$\tau_j e^1_3$};
  \draw (-0.4,0.3) node{$e^1_1$};
  \draw (-0.4,-0.3) node{$e^1_2$};
  \draw (0.5,-0.15) node{$e^1_3$};
  
  \fill[fill=black] (ti) circle (0.7pt);
  \fill[fill=black] (tj) circle (0.7pt);
  \fill[fill=black] (tk) circle (0.7pt);
  \fill[fill=black] (oi) circle (0.7pt);
  \fill[fill=black] (oj) circle (0.7pt);
  \fill[fill=black] (ok) circle (0.7pt);
  \fill[fill=black] (u) circle (0.7pt);
  \fill[fill=black] (o0) circle (0.7pt);
\end{tikzpicture}
\caption{The 1-skeleton of $\mathscr{D}$.}
\label{1skelO}
\end{figure}
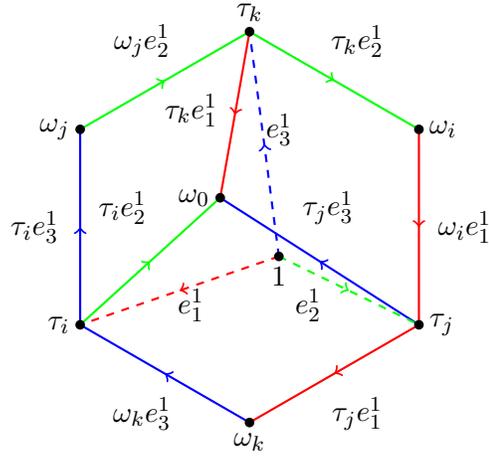
\end{center}
\indent Then, combining Proposition \ref{Ofunddom} and Lemma \ref{ODcells}, yields the following result:
\begin{prop}\label{DPOcells}
Letting $E^0:=\{1\}$, $E^1:=\{e^1_i,~i=1,2,3\}$, $E^2:=\{e^2_i,~i=1,2,3\}$ and $E^3:=\{e^3\}$ with the above notations, we have the following $\BO$-equivariant cellular decomposition of $\partial\Pol$
\[\partial\Pol=\coprod_{\substack{0\le j \le 3 \\ e\in E^j,g\in\BO}} ge.\]
As a consequence, using the homeomorphism $\phi : \partial\Pol \stackrel{\tiny{\sim}}\to \Sph^3$ given by $x \mapsto x/|x|$, we obtain the following $\BO$-equivariant cellular decomposition of the sphere
\[\Sph^3=\coprod_{\substack{0\le j \le 3 \\ e\in E^j,g\in\BO}} g\phi(e).\]
\end{prop}

We now have to compute the boundaries of the cells and the resulting cellular homology chain complex. We choose to orient the 3-cell $e^3$ directly, and the 2-cells undirectly. The induced orientations seen in $\mathscr{D}$ can be visualized in Figure \ref{figure55}.

\begin{center}
\begin{figure}[h!]
\begin{tikzpicture}[scale=2.6,rotate=90]
  \coordinate (ti) at (-0.5,0.866);
  \coordinate (tk) at (1,0);
  \coordinate (ok) at (-1,0);
  \coordinate (oi) at (0.5,-0.866);
  \coordinate (oj) at (0.5,0.866);
  \coordinate (tj) at (-0.5,-0.866);
  \coordinate (u) at (-0.15,-0.15);
  \coordinate (o0) at (0.15,0.15);
  
  \draw[pattern=north west lines,pattern color=green,opacity=0.5] (tj)--(oi)--(tk)--(u);
  \draw[pattern=north west lines,pattern color=red,opacity=0.5] (tj)--(u)--(ti)--(ok);
  \draw[pattern=north west lines,pattern color=blue,opacity=0.5] (ti)--(u)--(tk)--(oj);
  
  \draw[decoration={markings, mark=at position 0.625 with {\arrow{<}}},postaction={decorate}][rotate=0] (0.25,-0.433) ellipse (0.2 and 0.12);
  \draw[decoration={markings, mark=at position 0.625 with {\arrow{<}}},postaction={decorate}][rotate=0] (-0.5,0) ellipse (0.12 and 0.2);
  \draw[decoration={markings, mark=at position 0.625 with {\arrow{<}}},postaction={decorate}][rotate=0] (0.25,0.433) ellipse (0.2 and 0.12);
  
  \draw (0.25,-0.433) node{$e^2_1$};
  \draw (-0.5,0) node{$e^2_3$};
  \draw (0.25,0.433) node{$e^2_2$};
  
  \draw[thick] (oj)--(tk);
  \draw[thick] (tk)--(oi);
  \draw[thick] (oi)--(tj);
  \draw[thick] (tj)--(ok);
  \draw[thick] (ok)--(ti);
  \draw[thick] (ti)--(oj);
  \draw[thick] (ti)--(o0);
  \draw[thick] (tk)--(o0);
  \draw[thick] (tj)--(o0);
  \draw[dashed,opacity=0.45] (u)--(ti);
  \draw[dashed,opacity=0.45] (u)--(tj);
  \draw[dashed,opacity=0.45] (u)--(tk);
  
  \draw (u) node[below]{$1$};
  \draw (oi) node[right]{$\omega_i$};
  \draw (oj) node[left]{$\omega_j$};
  \draw (ok) node[below]{$\omega_k$};
  \draw (ti) node[left]{$\tau_i$};
  \draw (tj) node[right]{$\tau_j$};
  \draw (tk) node[above]{$\tau_k$};
  \draw (o0) node[right]{$\omega_0$};
  
  \fill[fill=black] (ti) circle (0.7pt);
  \fill[fill=black] (tj) circle (0.7pt);
  \fill[fill=black] (tk) circle (0.7pt);
  \fill[fill=black] (oi) circle (0.7pt);
  \fill[fill=black] (oj) circle (0.7pt);
  \fill[fill=black] (ok) circle (0.7pt);
  \fill[fill=black,opacity=0.45] (u) circle (0.7pt);
  \fill[fill=black] (o0) circle (0.7pt);
\end{tikzpicture}~~~~\begin{tikzpicture}[scale=2.5,rotate=90,thick]
  \coordinate (ti) at (-0.5,0.866);
  \coordinate (tk) at (1,0);
  \coordinate (ok) at (-1,0);
  \coordinate (oi) at (0.5,-0.866);
  \coordinate (oj) at (0.5,0.866);
  \coordinate (tj) at (-0.5,-0.866);
  \coordinate (u) at (-0.15,-0.15);
  \coordinate (o0) at (0.15,0.15);
  
  \fill[fill=red,opacity=0.7] (tj)--(oi)--(tk)--(o0);
  \fill[fill=blue,opacity=0.7] (tj)--(o0)--(ti)--(ok);
  \fill[fill=green,opacity=0.7] (ti)--(o0)--(tk)--(oj);
  
  \draw[decoration={markings, mark=at position 0.5 with {\arrow{<}}},postaction={decorate}][rotate=0] (0.3125,-0.483) ellipse (0.16 and 0.21);
  \draw[decoration={markings, mark=at position 0.5 with {\arrow{<}}},postaction={decorate}][rotate=0] (-0.4375,0.05) ellipse (0.21 and 0.16);
  \draw[decoration={markings, mark=at position 0.5 with {\arrow{<}}},postaction={decorate}][rotate=0] (0.3125,0.483) ellipse (0.16 and 0.21);
  
  \draw (0.3125,-0.483) node{$\tau_k e^2_3$};
  \draw (-0.4375,0.05) node{$\tau_j e^2_2$};
  \draw (0.3125,0.483) node{$\tau_i e^2_1$};
  
  \draw (oj)--(tk);
  \draw (tk)--(oi);
  \draw (oi)--(tj);
  \draw (tj)--(ok);
  \draw (ok)--(ti);
  \draw (ti)--(oj);
  \draw (ti)--(o0);
  \draw (tk)--(o0);
  \draw (tj)--(o0);
  
  \draw (oi) node[right]{$\omega_i$};
  \draw (oj) node[left]{$\omega_j$};
  \draw (ok) node[below]{$\omega_k$};
  \draw (ti) node[left]{$\tau_i$};
  \draw (tj) node[right]{$\tau_j$};
  \draw (tk) node[above]{$\tau_k$};
  \draw (o0) node[right]{$\omega_0$};
  
  \fill[fill=black] (ti) circle (0.7pt);
  \fill[fill=black] (tj) circle (0.7pt);
  \fill[fill=black] (tk) circle (0.7pt);
  \fill[fill=black] (oi) circle (0.7pt);
  \fill[fill=black] (oj) circle (0.7pt);
  \fill[fill=black] (ok) circle (0.7pt);
  \fill[fill=black] (o0) circle (0.7pt);
\end{tikzpicture}
\caption{The fundamental domain with its 2-cells (back and front).}
\label{figure55}
\end{figure}
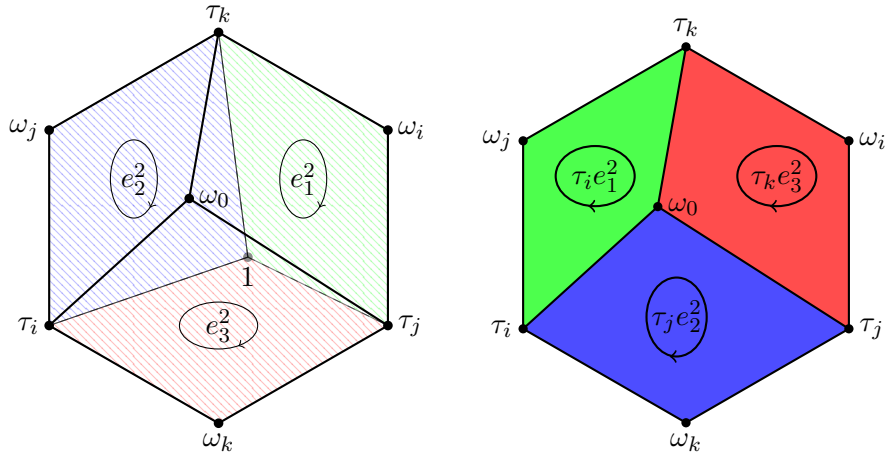
\end{center}
\indent These orientations allow us to easily compute the boundaries of the representing cells $e^u_v$ and give the resulting chain complex of free left $\Z[\BO]$-modules. 
\begin{prop}\label{Ochain}
The cellular homology complex of $\partial\Pol$ associated to the cellular structure given in Proposition \ref{DPOcells} is the chain complex of left $\Z[\BO]$-modules
\[\mathcal{K}_\BO := \left(\xymatrix{\Z[\BO] \ar^{\partial_3}[r] & \Z[\BO]^3 \ar^{\partial_2}[r] & \Z[\BO]^3 \ar^{\partial_1}[r] & \Z[\BO]}\right),\]
where
\[\partial_1=\begin{pmatrix} \tau_i-1 \\ \tau_j-1 \\ \tau_k-1\end{pmatrix},~~~~\partial_2=\begin{pmatrix} \omega_i & \tau_k-1 & 1 \\ 1 & \omega_j & \tau_i-1 \\ \tau_j-1 & 1 & \omega_k\end{pmatrix},~~~~\partial_3=\begin{pmatrix} 1-\tau_i & 1-\tau_j & 1-\tau_k \end{pmatrix}.\]
\end{prop}
\indent To conclude this section, we show in Figure \ref{octarep} a tetrahedron in $\Pol_3$ containing $1$ as a vertex. In this picture, we put the points $\omega_h^\pm$ (with $h=0,i,j,k$) at the centers of the facets of the octahedron\footnote{The points in gray are on the background of the figure}. The tetrahedra in question are constructed in the following way: one chooses an edge of the octahedron and the center of a face which is adjacent to this edge. The resulting four vertices (including $1$) are vertices of the corresponding tetrahedron. 
\newline
\indent This representation will be useful when we study the application to the flag manifold of $SL_3(\R)$.

\begin{center}
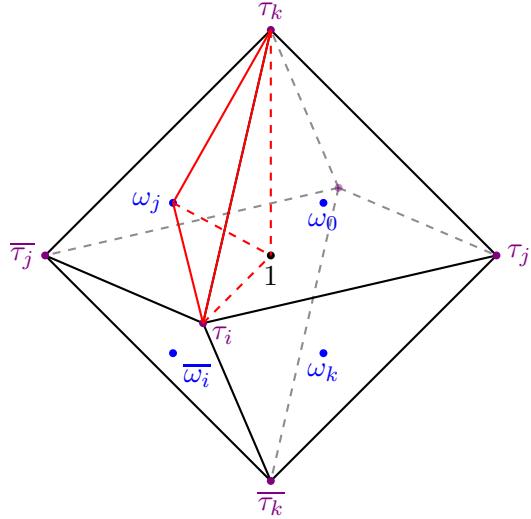
\begin{figure}[h!]
\begin{tikzpicture}[scale=3,thick]
  \coordinate (u) at (0,0,0);
  \coordinate (tip) at (0,0,0.78);
  \coordinate (tim) at (0,0,-0.78);
  \coordinate (tjp) at (1,0,0);
  \coordinate (tjm) at (-1,0,0);
  \coordinate (tkp) at (0,1,0);
  \coordinate (tkm) at (0,-1,0);
  \coordinate (o0p) at (1/3,1/3,0.26);
  \coordinate (o0m) at (-1/3,-1/3,-0.26);
  \coordinate (ojp) at (-1/3,1/3,0.26);
  \coordinate (ojm) at (1/3,-1/3,-0.26);
  \coordinate (oim) at (-1/3,-1/3,0.26);
  \coordinate (oip) at (1/3,1/3,-0.26);
  \coordinate (okp) at (1/3,-1/3,0.26);
  \coordinate (okm) at (-1/3,1/3,-0.26);
  
  \draw (tip)--(tjp);
  \draw (tip)--(tjm);
  \draw (tip)--(tkp);
  \draw (tip)--(tkm);
  \draw[opacity=0.45,dashed] (tim)--(tjp);
  \draw[opacity=0.45,dashed] (tim)--(tjm);
  \draw[opacity=0.45,dashed] (tim)--(tkp);
  \draw[opacity=0.45,dashed] (tim)--(tkm);
  \draw (tjm)--(tkm);
  \draw (tkm)--(tjp);
  \draw (tjp)--(tkp);
  \draw (tkp)--(tjm);
  
  \fill[fill=black] (u) circle (0.5pt);
  \fill[fill=black,color=violet] (tip) circle (0.5pt);
  \fill[fill=black,opacity=0.45,color=violet] (tim) circle (0.5pt);
  \fill[fill=black,color=violet] (tjp) circle (0.5pt);
  \fill[fill=black,color=violet] (tjm) circle (0.5pt);
  \fill[fill=black,color=violet] (tkp) circle (0.5pt);
  \fill[fill=black,color=violet] (tkm) circle (0.5pt);
  \fill[fill=black,color=blue] (o0p) circle (0.5pt);
  \fill[fill=black,color=blue] (oim) circle (0.5pt);
  \fill[fill=black,color=blue] (ojp) circle (0.5pt);
  \fill[fill=black,color=blue] (okp) circle (0.5pt);
  
  \draw (u) node[below]{$1$};
  \draw[color=violet] (-0.02,0.02,0.75) node[below right]{$\tau_i$};
  \draw[color=violet] (tjp) node[right]{$\tau_j$};
  \draw[color=violet] (tjm) node[left]{$\overline{\tau_j}$};
  \draw[color=violet] (tkp) node[above]{$\tau_k$};
  \draw[color=violet] (tkm) node[below]{$\overline{\tau_k}$};
  \draw[color=blue] (o0p) node[below]{$\omega_0$};
  \draw[color=blue] (oim) node[below right]{$\overline{\omega_i}$};
  \draw[color=blue] (ojp) node[left]{$\omega_j$};
  \draw[color=blue] (okp) node[below]{$\omega_k$};
  
  \draw[color=red] (tip)--(tkp);
  \draw[color=red] (tkp)--(ojp);
  \draw[color=red] (ojp)--(tip);
  \draw[dashed,color=red] (tkp)--(u);
  \draw[dashed,color=red] (ojp)--(u);
  \draw[dashed,color=red] (tip)--(u);
\end{tikzpicture}
\caption{One of the twenty-four facets of $\Pol$ containing $1$.}
\label{octarep}
\end{figure}
\end{center}

\subsection{The case of spheres and free resolution of the trivial $\BO$-module}\label{3.3}
\hfill\\

Using Theorem \ref{thm6.4}, we derive a fundamental domain for $\BO$ acting on $\Sph^3$ and thus obtain an $\BO$-equivariant cellular decomposition of $\Sph^3$.

\begin{theo}\label{SPHOcells}
The following subset of $\Sph^3$ is a fundamental domain for the action of $\BO$
\begin{align*}
\mathscr{F}_{3}:=&(\omega_i\ast 1 \ast \tau_j\ast\tau_k)\cup(1\ast\tau_j\ast\tau_k\ast\omega_0)\cup(\omega_j\ast1\ast\tau_k\ast\tau_i) \\
&\cup(1\ast\tau_k\ast\tau_i\ast\omega_0)\cup(\omega_k\ast1\ast\tau_i\ast\tau_j)\cup(1\ast\tau_i\ast\tau_j\ast\omega_0).
\end{align*}

As a consequence, the sphere $\Sph^3$ admits a $\BO$-equivariant cellular decomposition with the following cells as orbit representatives, where $\mathrm{relint}$ denotes the relative interior,
\[\widetilde{e}^0:=1\ast\emptyset=\{1\},~\widetilde{e}^1_1:=\mathrm{relint}(1\ast\tau_i),~\widetilde{e}^1_2:=\mathrm{relint}(1\ast\tau_j),~\widetilde{e}^1_3:=\mathrm{relint}(1\ast\tau_k),\]
\[\widetilde{e}^2_1:=\mathrm{relint}((1\ast\omega_i\ast\tau_j)\cup(1\ast\omega_i\ast\tau_k)),~\widetilde{e}^2_2:=\mathrm{relint}((1\ast\omega_j\ast\tau_k)\cup(1\ast\omega_j\ast\tau_i)),\]
\[\widetilde{e}^2_3:=\mathrm{relint}((1\ast\omega_k\ast\tau_i)\cup(1\ast\omega_k\ast\tau_j)),~\widetilde{e}^3:=\interior{\mathscr{F}}_{3}\]

Furthermore, the associated cellular homology complex is the chain complex $\mathcal{K}_\BO$ from the Proposition \ref{Ochain}.
\end{theo}

The relative interior of a simplex in $\partial\Pol$ is sent to the Riemannian relative interior in $\Sph^3$, that is, the subset of points of the geodesic simplex which do not belong to any geodesic sub-simplex of smaller dimension.

From this, we immediately deduce the following result, which gives a periodic free resolution of the constant module over $\Z[\BO]$.

\begin{cor}\label{resolforO}
The following complex is a 4-periodic resolution of $\Z$ over $\Z[\BO]$
\[\xymatrix{\dotsc \ar[r] & \Z[\BO]^3 \ar^{\partial_{4q-3}}[r] & \Z[\BO] \ar^{\partial_{4q-4}}[r] & \dotsc\ar[r] & \Z[\BO]^3 \ar^{\partial_2}[r] & \Z[\BO]^3 \ar^<<<<<{\partial_1}[r] & \Z[\BO] \ar^<<<<<{\varepsilon}[r] & \Z \ar[r] & 0},\]
where, for $q\ge1$,
\[\partial_{4q-3}=\begin{pmatrix} \tau_i-1 \\ \tau_j-1 \\ \tau_k-1\end{pmatrix},~~~~\partial_{4q-2}=\begin{pmatrix} \omega_i & \tau_k-1 & 1 \\ 1 & \omega_j & \tau_i-1 \\ \tau_j-1 & 1 & \omega_k\end{pmatrix},\]
\[\partial_{4q-1}=\begin{pmatrix} 1-\tau_i & 1-\tau_j & 1-\tau_k \end{pmatrix},~~~~\partial_{4q}=\begin{pmatrix}\sum_{g\in\BO}g\end{pmatrix}.\]
\end{cor}

Recall the augmentation map $\varepsilon : \Z[\BO] \to \Z$ defined by $\varepsilon\left(\sum_{g\in\BO}a_gg\right):=\sum_{g\in\BO}a_g$. We can now compute the group cohomology of $\BO$ using the previous Corollary. But first, let us recall the following basic fact:
\begin{lem}\label{quotient_complex}
\begin{enumerate}
\item If $G$ is a finite group acting freely and cellularily on a CW-complex $X$ and $\mathcal{K}$ is the cellular homology chain complex of $X$ (a complex of free $\Z[G]$-modules), then the induced cellular homology complex of $X/G$ is $\mathcal{K}\otimes_{\Z[G]}\Z$.
\item If $f : \Z[G]^m \to \Z[G]^n$ is a homomorphism of left $\Z[G]$-modules, identified with its matrix in the canonical bases, then the matrix of the induced homomorphism $f\otimes_{\Z[G]}id_\Z : \Z^m \to \Z^n$ is given by the matrix $\varepsilon(f)$, computed term by term.
\end{enumerate}
\end{lem}
\begin{proof}
The first statement is obvious, by definition of the cellular structure on $X/G$ and the second one is a direct calculation.
\end{proof}
\begin{cor}\label{HstarO}
The group cohomology of $\BO$ with integer coefficients is given as follows:
\[\forall q\ge 1,~\left\{\begin{array}{cc}
H^q(\BO,\Z)=\Z & \text{if}~~q=0, \\[.5em]
H^q(\BO,\Z)=\Z/48\Z & \text{if}~~q\equiv 0\pmod 4, \\[.5em]
H^q(\BO,\Z)=\Z/2\Z & \text{if}~~q\equiv 2 \pmod 4, \\[.5em]
H^q(\BO,\Z)=0 & \text{otherwise}.\end{array}\right.\]
\end{cor}
\begin{proof}
In view of Lemma \ref{quotient_complex}, is suffices to compute $\mathcal{C}(\mathsf{P}^{\infty}_{\BO},\Z[\BO])\otimes_{\Z[\BO]}\Z$, with $\mathcal{C}(\mathsf{P}_{\BO}^\infty,\Z[\BO])$ the complex given in Corollary \ref{resolforO}. The notation will become clear later (see Theorem \ref{SPHNOcells}). Computing the matrices $\varepsilon(\partial_i)$ and dualizing the result leads to the following cochain complex
\[\xymatrix{0\ar[r] & \Z \ar^0[r] & \Z^3 \ar^{\left(\begin{smallmatrix}1 & 1 & 0 \\ 0 & 1 & 1 \\ 1 & 0 & 1\end{smallmatrix}\right)}[r] & \Z^3 \ar^0[r] & \Z \ar^{\times48}[r] & \Z \ar[r] & \cdots\ar[r] & \Z\ar^{\times48}[r] & \Z \ar^0[r] & \Z^3 \ar[r] & \cdots}\]
and computing the elementary divisors of the only non-trivial matrix concludes.
\end{proof}
\begin{rem}\label{TZO}
In \cite[Proposition 4.7]{tomoda-zvengrowski_2008}, Tomoda and Zvengrowski give an explicit resolution of $\Z$ over $\Z[\BO]$. They use the following presentation 
\[\BO=\left<T,U~|~TU^2T=U^2,~TUT=UTU\right>\]
from \cite{coxeter-moser}. As we would like to work with presentations, we use the isomorphism
\[\left<T,U~|~TU^2T=U^2,~TUT=UTU\right>  \stackrel{\tiny{\sim}}\longrightarrow \BO\]
sending $T$ to $\tfrac{1}{\sqrt{2}}(1+i)$ and $U$ to $\tfrac{1}{\sqrt{2}}(1+j)$.
Then, the Tomoda-Zvengrowski complex reads
\[\mathcal{K}_\BO^{\text{TZ}}=\left(\xymatrix{\Z[\BO] \ar^{\delta_3}[r] & \Z[\BO]^2 \ar^{\delta_2}[r] & \Z[\BO]^2 \ar^{\delta_1}[r] & \Z[\BO]}\right),\]
with
\[\delta_1=\begin{pmatrix}T-1 \\ U-1 \end{pmatrix},~~\delta_2=\begin{pmatrix}1+TU-U & T-1-UT \\ 1+TU^2 & T-U-1+TU\end{pmatrix},~~\delta_3=\begin{pmatrix}1-TU & U-1\end{pmatrix}.\]
On the other hand, the differentials $\partial_i$ of the complex $\mathcal{K}_\BO$ from Proposition \ref{Ochain} are given, through the above presentation, by
\[\partial_1=\begin{pmatrix}T-1 \\ U-1 \\ TUT^{-1}-1\end{pmatrix},~\partial_2=\begin{pmatrix}UT^{-1} & TUT^{-1}-1 & 1 \\ 1 & U^{-1}T & T-1 \\ U-1 & 1 & UT\end{pmatrix},~\partial_3=\begin{pmatrix}1-T & 1-U & 1-TUT^{-1}\end{pmatrix}.\]
We claim that the complexes $\mathcal{K}_\BO$ and $\mathcal{K}_\BO^{\text{TZ}}$ are homotopy equivalent. This observation relies on elementary operations on matrix rows and columns. Write $Z:=U^4=T^4$ for the only non trivial element of $Z(\BO)$. For short, define
\[P:=\begin{pmatrix}-Z & 0 & 0 \\ Z(1-T) & TUT & -U^2 \\ -U^{-3}T & -TUT & 0\end{pmatrix},~~Q:=\begin{pmatrix}0 & -TUT & 0 \\ -TUT & 0 & 0 \\ U^2-TUT & U^2T & 1 \end{pmatrix},\]
then $P,Q\in GL_3(\Z[\BO])$ and 
\[P^{-1}=\begin{pmatrix}-Z & 0 & 0 \\ U^{-1} & 0 & -(TUT)^{-1} \\ U^{-2}(T-1)+U^{-1}T & -U^{-2} & -U^{-2}\end{pmatrix},~~Q^{-1}=\begin{pmatrix}0 & -(TUT)^{-1} & 0 \\ -(TUT)^{-1} & 0 & 0 \\ UT^{-1} & TUT^{-1}-1 & 1\end{pmatrix}.\]
Now, we have the following relations
\[-Q^{-1}d_1TUT=\begin{pmatrix}T-1 \\ U-1 \\ 0\end{pmatrix},~~P^{-1}d_2Q=\begin{pmatrix}0 & 0 & -Z \\ 1+TU-U & T-1-UT & 0 \\ 1+TU^2 & T-U-1+TU & 0\end{pmatrix},\]
\[U^{-2}d_3P=\begin{pmatrix}0 & 1-TU & U-1\end{pmatrix}.\]
Hence, the isomorphism
\[\mathcal{K}_\BO\simeq\mathcal{K}_\BO^{\text{TZ}}\oplus\left(\xymatrix{0 \ar[r] & \Z[\BO] \ar^1[r] & \Z[\BO] \ar[r] & 0}\right),\]
confirms that $\mathcal{K}_\BO$ is indeed homotopy equivalent to $\mathcal{K}_\BO^{\text{TZ}}$.
\end{rem}

\indent In fact, the complex from the Corollary \ref{resolforO} carries geometric information. 

\begin{prop}\label{OfunddomN}
The following subset of $\Sph^{4n-1}$ is a fundamental domain for the action
\[\mathscr{F}_{4n-1}:=\Sigma_1\ast\Sigma_2\ast\cdots\ast\Sigma_{2(n-1)}\ast\mathscr{F}_{3},\]
with $\mathscr{F}_{3}$ inside $\Sigma_{2n-1}\ast\Sigma_{2n}$ the fundamental domain from Theorem \ref{SPHOcells}.
\end{prop}
\indent We can now describe the resulting equivariant cellular decomposition on $\Sph^{4n-1}$ using Lemma \ref{lem_join} and Theorem \ref{SPHOcells}. It only remains to consider the boundary of the cells $\widetilde{e}^{4q}$ for $q>0$. But it follows from the fact that $\widetilde{e}^{4q}=\Sph^{4q-1}\ast\widetilde{e}^{4q-1}$, hence its boundary is given by all the cells in $\Sph^{4q-1}$, that is, all the orbits under $\BO$. This gives the following result, which we prefer to state using the vocabulary of universal covering spaces. We denote by $C(\widetilde{K},\Z[G])$ the chain complex of finitely generated free (left) $\Z[G]$-modules given by the cellular homology complex of the universal covering space $\widetilde{K}$ of a finite CW-complex $K$ with the fundamental group $G$ acting by covering transformations.
\begin{theo}\label{SPHNOcells}
The chain complex $\mathcal{C}(\mathsf{P}_{\BO}^{4n-1},\Z[\BO])$ of the universal covering space of the octahedral space forms $\mathsf{P}_{\BO}^{4n-1}$ with the fundamental group acting by covering transformations is the following complex of left $\Z[\BO]$-modules:
\[\xymatrix{0 \ar[r] & \Z[\BO] \ar^{\partial_{4n-1}}[r] & \Z[\BO]^3 \ar[r] & \dotsc\ar[r] & \Z[\BO]^3 \ar^{\partial_2}[r] & \Z[\BO]^3 \ar^{\partial_1}[r] & \Z[\BO] \ar[r] & 0},\]
where, the boundaries are as in Corollary \ref{resolforO}.

In particular, the complex is exact in middle terms, i.e.
\[\forall 0<i<4n-1,~H_i(\mathcal{C}(\mathsf{P}_{\BO}^{4n-1},\Z[\BO]))=0\]
and we have
\[H_0(\mathcal{C}(\mathsf{P}_{\BO}^{4n-1},\Z[\BO]))=H_{4n-1}(\mathcal{C}(\mathsf{P}_{\BO}^{4n-1},\Z[\BO]))=\Z.\]
\end{theo}
\begin{proof}
The computation of the complex follows from lemma \ref{lem_join} and the previous discussion. The claims on its homology follow, $\Sph^{4n-1}$ being the universal covering space of $\mathsf{P}^{4n-1}_{\BO}$.
\end{proof}

\subsection{Application to the flag manifold of $SL_3(\R)$}
\hfill\\

The $\BO$-equivariant cellular structure of $\Sph^3$ may be used to obtain a cellular decomposition of the real points of the flag manifold $SU_3(\C)/T$ of type $A_2$. The elementary facts concerning Lie groups we use here can be found in \cite{bump} or \cite{fulton_harris}.
\newline
\indent Given a \emph{maximal torus} $T$ in a simply connected compact semisimple Lie group $G$, one can consider the \emph{Weyl group} $W:=N_G(T)/T$. It is a finite Coxeter group (\cite[Proposition 15.8 and Theorem 25.1]{bump}), which acts by right multiplication on the \emph{flag manifold} $G/T$. For instance, in type $A_{n-1}$, we have $G=SU_n(\C)$ and we can take $T$ to be the group of diagonal matrices in $SU_n(\C)$. In this case, one has $W\simeq\Sym_n$. This group has Coxeter presentation
\[W=\Sym_n=\left<s_1,\dotsc,s_{n-1}~|~s_i^2=1,~s_is_{i+1}s_i=s_{i+1}s_is_{i+1},~s_is_j=s_js_i,~\forall |i-j|>1\right>\]
and a representative $\dot{s_i}$ for the \emph{reflection} $s_i$ in $N_{SU_n(\C)}(T)$ can be taken as a block matrix (with $(i-1)$ ones before the matrix $s$):
\[\dot{s_i}:=\mathrm{diag}(1,\dotsc,1,s,1\dotsc,1),~~\text{with}~~s:=\begin{pmatrix}0 & -1\\1 & 0\end{pmatrix}.\]
If $w=s_{i_1}s_{i_2}\cdots s_{i_k}$ is a \emph{reduced word} in $W$, then the element $\dot{w}:=\dot{s_{i_1}}\dot{s_{i_2}}\cdots\dot{s_{i_k}}\in N_G(T)$ does not depend on the chosen word for $w$ and for $g\in G$, the action of $w$ on $g$ is given by multiplication $g\cdot w:=g\dot{w}$.
\newline
\indent On the other hand, the \emph{Iwasawa decomposition} (see \cite[Theorem 26.4]{bump}) gives a diffeomorphism $G/T\simeq G^\C/B$, with $G^\C$ the \emph{universal complexification} of $G$ and $B$ a \emph{Borel subgroup} of $G^\C$ containing $T$. This provides $G/T$ with a structure of complex algebraic variety. Hence, one may talk about \emph{real points} of $G/T$. We use the standard notation $X(\R)$ to denote the set of real points of an algebraic variety $X$.
\begin{rem}\label{flags}
In type $A_{n-1}$, that is if $G=SU_n(\C)$ and if $T$ is the group diagonal matrices in $SU_n(\C)$, then one may take $G^{\C}=SL_n(\C)$ and $B$ the Borel subgroup of upper-triangular matrices in $SL_n(\C)$. We denote by $\mathcal{F}_n$ the set of \emph{flags} in $\C^n$, that is
\[\mathcal{F}_n:=\{V_\bullet:=(V_1,\dotsc,V_{n-1})~;~V_i\le\C^n,~V_i\subset V_{i+1},~\dim V_i=i\}.\]
The group $G^\C$ acts naturally on $\mathcal{F}_n$ and if $V_0$ is the canonical flag of $\C^n$, then the bijection
\[\begin{array}{ccc}
G^\C/B & \to & \mathcal{F}_n \\ gB & \mapsto & g\cdot V_0\end{array}\]
endows $\mathcal{F}_n$ with the structure of a complex algebraic variety. Furthermore, it is easy to see that the real points $\mathcal{F}_n(\R)$ of $\mathcal{F}_n$ is the set of \emph{real flags} in $\R^n$ and we have
\[\mathcal{F}_n(\R)\simeq SO_n(\R)/T(\R)\]
and $T(\R)$ is isomorphic to $(\Z/2\Z)^{n-1}$.
\end{rem}
\indent The case $G=SU_2(\C)$ (i.e. in type $A_1$) is fairly trivial, since $SU_2(\C)/T\simeq\Sph^2$ and $W=\Sym_2=\{1,s\}$ acts as the antipode on $\Sph^2$, so the quotient $(SU_2(\C)/T)/\Sym_2$ is the projective plane $\Pro^2(\R)$ and its simplest cellular structure lifts to a $W$-equivariant one on $\Sph^2$, see Figure \ref{A1}.
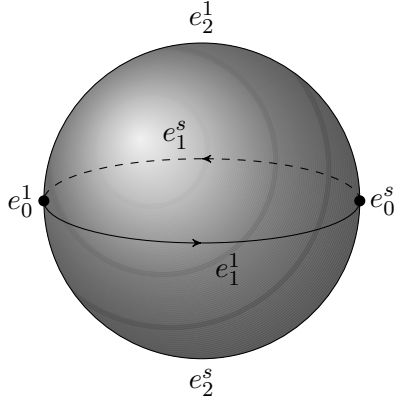
\begin{figure}[h!]
\begin{tikzpicture}[scale=0.7]
  \shade[ball color = black!40, opacity = 0.5] (0,0) circle (3cm);
  \draw (0,0) circle (3cm);
  \draw[->,>=stealth'] (-3,0) arc (180:270:3 and 0.8);
  \draw (0,-0.8) arc (270:360:3 and 0.8);
  \draw[->,>=stealth'][dashed] (3,0) arc (0:90:3 and 0.8);
  \draw[dashed] (0,0.8) arc (90:180:3 and 0.8);
  \fill[fill=black] (-3,0) circle (3pt);
  \draw (-3,0) node[left]{$e_0^1$};
  \fill[fill=black] (3,0)  circle (3pt);
  \draw (3,0) node[right]{$e_ 0^s$};
  \draw (0.5,-0.8) node[below]{$e_1^1$};
  \draw (-0.5,0.8) node[above]{$e_1^s$};
  \draw (0,3) node[above]{$e_2^1$};
  \draw (0,-3) node[below]{$e_2^s$};
\end{tikzpicture}
\caption{Equivariant cellular decomposition of $SU_2(\C)/T=\Sph^2$.}
\label{A1}
\end{figure}
\newline
\indent In this section, we treat the case of the real points of $SU_3(\C)/T$, using the octahedral spherical space form.
\newline
\indent First of all, we have to identify spaces and actions. We begin with a trivial lemma.
\begin{lem}\label{quotient_action}
Let $P$ be a finite group acting freely by diffeomorphisms on a manifold $X$ and $Q\unlhd P$ be a normal subgroup of $P$. Then, $P/Q$ acts freely on the quotient manifold $X/Q$ and the projection $X\twoheadrightarrow X/P$ induces a natural diffeomorphism
\[\bigslant{(X/Q)}{(P/Q)} \stackrel{\tiny{\sim}}\longrightarrow X/P.\]
\end{lem}
\indent We will apply this lemma to $P=\BO$, $Q=\mathcal{Q}_8$ and $X=\Sph^3$. One has to be careful at this point: we let $\BO$ act on $\Sph^3$ on the \emph{left}, whereas $W=\Sym_3$ naturally acts on $\mathcal{F}(\R)$ on the \emph{right}. Hence we let $\BO$ act on the right on $\Sph^3$ by multiplication. It is straightforward to adapt our results to this case. For instance, we replace $\Delta_i=:\conv(q_1,q_2,q_3,q_4)$ by $\widehat{\Delta_i}:=\conv(q_1^{-1},q_2^{-1},q_3^{-1},q_4^{-1})$ and $\mathscr{F}_{3}$ by $\widehat{\mathscr{F}_{3}}:=\mathrm{pr}(\widehat{\mathscr{D}})$ where $\mathrm{pr}(x)={x}/{|x|}$ is the usual projection and $\widehat{\mathscr{D}}:=\bigcup_i\widehat{\Delta_i}$ and we can do the same for the cells in $\Sph^3$. Briefly, we just have to replace every quaternion appearing in sections \ref{3.1}, \ref{3.2} and \ref{3.3} by its inverse and left multiplications by right multiplications.
\newline
\indent Now, denoting by $\mathcal{F}:=SU_3(\C)/T\simeq SL_3(\C)/B$ the flag manifold, we have a diffeomorphism
\[\mathcal{F}(\R)\simeq SO_3(\R)/T(\R).\]
\indent Recall the surjective homomorphism $\mathrm{B} : \Sph^3 \twoheadrightarrow SO_3(\R)$, with kernel $\{\pm1\}$. We have a surjective homomorphism
\[\widetilde{\phi} : \Sph^3 \stackrel{\tiny{\mathrm{B}}}\twoheadrightarrow SO_3(\R) \twoheadrightarrow SO_3(\R)/T(\R)\simeq \mathcal{F}(\R).\]
Now, it is clear that $\mathrm{B}^{-1}(T(\R))=\{\pm1,\pm i,\pm j,\pm k\}=\mathcal{Q}_8$. The lemma \ref{quotient_action} applied to $G=\mathcal{Q}_8$, $N:=\{\pm1\}=Z(\mathcal{Q}_8)$ and $X=\Sph^3$ leads to the following result:
\begin{lem}\label{ident_spaces}
Denoting by $\mathcal{F}:=SU_3(\C)/T$ the flag manifold of type $A_2$, the above defined map $\widetilde{\phi}$ induces a diffeomorphism
\[\phi : \Sph^3/\mathcal{Q}_8 \stackrel{\tiny{\sim}}\longrightarrow \mathcal{F}(\R).\]
\end{lem}
\indent Now, one has $W=\Sym_3=\left<s_\alpha,s_\beta~|~s_\alpha^2=s_\beta^2=1,~s_\alpha s_\beta s_\alpha=s_\beta s_\alpha s_\beta\right>$ (the notation $s_\alpha,~s_\beta$ makes reference to the simple roots $\alpha$ and $\beta$ of the root system of type $A_2$). The reflections $s_\alpha$ ans $s_\beta$ can be represented in $SO_3(\R)$ by the following matrices
\[\dot{s_\alpha}=\left(\begin{smallmatrix}0 & -1 & 0 \\ 1 & 0 & 0 \\ 0 & 0 & 1\end{smallmatrix}\right),~\dot{s_\beta}=\left(\begin{smallmatrix}1 & 0 & 0 \\ 0 & 0 & -1 \\ 0 & 1 & 0\end{smallmatrix}\right).\]
These matrices may be obtained from $\Sph^3$ using $\mathrm{B}$:
\[\dot{s_\alpha}=\mathrm{B}\left(\frac{1+k}{\sqrt{2}}\right),~\dot{s_\beta}=\mathrm{B}\left(\frac{1+i}{\sqrt{2}}\right),\]
and this induces a well-defined isomorphism
\[\begin{array}{ccccc}\sigma & : & \BO/\mathcal{Q}_8 & \stackrel{\tiny{\sim}}\longrightarrow & \Sym_3 \\ & & (1+i)/\sqrt{2} & \longmapsto & s_\beta \\ & & (1+k)/\sqrt{2} & \longmapsto & s_\alpha\end{array}\]
Therefore, recalling that $\Sym_3=N_{SU_3(\C)}(T)/T=(N_{SO_3(\R)}(SO_3(\R)\cap T))/(SO_3(\R)\cap T)$ acts on $\mathcal{F}(\R)$ by multiplication on the right by a representative matrix, one obtains the following relation
\[\forall (x,g)\in\Sph^3\times\BO,~\phi(\overline{x})\cdot\sigma(\overline{g})=\widetilde{\phi}(xg).\]
\indent Henceforth, using the lemma \ref{quotient_action}, one obtains the following result:
\begin{prop}\label{ident_actions}
The diffeomorphism $\psi$ from the Lemma \ref{ident_spaces} induces a diffeomorphism
\[\overline{\phi} : \Sph^3/\BO \stackrel{\tiny{\sim}}\longrightarrow \mathcal{F}(\R)/\Sym_3.\]
\newline
\indent In particular, $\BO$-equivariant cellular structure on $\Sph^3$ defined in Theorem \ref{SPHOcells} induces an $\Sym_3$-equivariant cellular structure on the real flag manifold $\mathcal{F}(\R)$.
\end{prop}
\begin{cor}\label{pi1}
The fundamental groups of the real flag manifold $\mathcal{F}(\R)$ and of its quotient space by $\Sym_3$ are given by
\[\pi_1(\mathcal{F}(\R),\ast)=\mathcal{Q}_8~~~~\text{and}~~~~\pi_1(\mathcal{F}(\R)/\Sym_3,\ast)=\BO.\]
\end{cor}
\indent We are now in a position to state and prove the principal result of this section:
\begin{theo}\label{S3equiv}
The real flag manifold $\mathcal{F}(\R)=SO_3(\R)/T(\R)$ admits an $\Sym_3$-equivariant cellular decomposition with orbit representatives cells given by
\[\mathfrak{e}^i_j:=\phi\left(\pi_{\mathcal{Q}_8}\left((e^i_j)^{-1}\right)\right),\]
where $\pi_{\mathcal{Q}_8} : \Sph^3\to\Sph^3/\mathcal{Q}_8$ is the natural projection, $\phi : \Sph^3/\mathcal{Q}_8 \to \mathcal{F}(\R)$ is the $\Sym_3$-equivariant diffeomorphism from the Proposition \ref{ident_spaces} and $e^i_j$ are the cells of the $\BO$-equivariant cellular decomposition from the Theorem \ref{SPHOcells}.
\newline 
\indent Furthermore, the associated cellular homology complex is the chain complex of free right $\Z[\Sym_3]$-modules
\[\mathcal{K}_{\Sym_3}:=\left(\xymatrix{\Z[\Sym_3] \ar^{\partial_3}[r] & \Z[\Sym_3]^3 \ar^{\partial_2}[r] & \Z[\Sym_3]^3 \ar^{\partial_1}[r] & \Z[\Sym_3]}\right),\]
where
\[\partial_1=\begin{pmatrix}1-s_\beta & 1-w_0 & 1-s_\alpha\end{pmatrix},~~~~\partial_2=\begin{pmatrix}s_\alpha s_\beta & 1 & w_0-1 \\ s_\alpha-1 & s_\alpha s_\beta & 1 \\ 1 & s_\beta-1 & s_\alpha s_\beta\end{pmatrix},~~~~\partial_3=\begin{pmatrix}1-s_\beta \\ 1-w_0 \\ 1-s_\alpha\end{pmatrix}.\]
\end{theo}
\begin{proof}
This only relies on Proposition \ref{ident_actions} and the fact that $((e^i_j)^{-1})_{i,j}$ is an $\BO$-equivariant cell decomposition of $\Sph^3$, the group $\BO$ acting by right multiplication on the sphere. Next, we have to determine the images of the points of $\BO$ we used to construct $\widehat{\mathscr{F}_{\BO,3}}$ under the projection 
\[\pi^\BO:\BO\twoheadrightarrow\BO/\mathcal{Q}_8\stackrel{\tiny{\sigma}}\simeq\Sym_3.\]
Recall that, denoting by $s_\alpha$ and $s_\beta$ the simple reflections in the Weyl group $W=\Sym_3$, we have
\[\Sym_3=\left<s_\alpha,s_\beta~|~s_\alpha^2=s_\beta^2=1,~s_\alpha s_\beta s_\alpha=s_\beta s_\alpha s_\beta\right>=\{1,s_\alpha,s_\beta,s_\alpha s_\beta,s_\beta s_\alpha,s_\alpha s_\beta s_\alpha\}\]
and we denote by $w_0:=s_\alpha s_\beta s_\alpha$ the \emph{longest element} of $\Sym_3$. We compute
\[\tau_i\mapsto s_\beta,~\tau_j\mapsto w_0,~\tau_k\mapsto s_\alpha,~\omega_i,\omega_j,\omega_k\mapsto s_\beta s_\alpha,~\omega_0\mapsto s_\alpha s_\beta.\]
Thus, the resulting cellular homology chain complex can be computed from the one in Theorem \ref{SPHOcells}, replacing each coefficient $q\in\BO$ in $\partial_i$ by $\pi^\BO(q^{-1})$ and transposing the matrices.
\end{proof}

We can now deduce the action of $\Sym_3$ on the cohomology of $\mathcal{F}(\R)$. Since $\Sym_3$ acts on the right of $\mathcal{F}(\R)$ and since cohomology is a contravariant functor, $\Sym_3$ acts on the left on $H^*(\mathcal{F}(\R),\Z)$.
\newline
\indent First of all, define the integral representation
\[\mathbf{2} : \Sym_3 \to GL_2(\Z)\]
by
\[\mathbf{2}(s_\alpha)=\begin{pmatrix}0 & 1\\1 & 0\end{pmatrix},~~\mathbf{2}(s_\beta)=\begin{pmatrix}1 & 0\\-1 & -1\end{pmatrix}.\]
Then, $\mathbf{2}$ is an integral form of the 2-dimensional irreducible complex representation of $\Sym_3$. Its reduction modulo 2 is the irreducible $\mathbb{F}_2[\Sym_3]$-module ${\mathbf{2}}\otimes\mathbb{F}_2$ of dimension 2. Moreover, we let $\overline{\mathbf{2}}$ be the representation $\Z[\Sym_3] \to \mathrm{End}_\Z(\mathbb{F}_2^2)$.
\newline
\indent For convenience, we consider $\Z[\Sym_3]$ as a graded algebra concentrated in degree zero.
\begin{cor}\label{S3actioncohomology}
The cohomology $H^*(\mathcal{F}(\R),\Z)$ of $\mathcal{F}(\R)$ is a graded commutative left $\Z[\Sym_3]$-module such that
\[H^i(\mathcal{F}(\R),\Z)=\left\{\begin{array}{cc}\mathds{1} & \text{if }i=0,3, \\ \overline{\mathbf{2}} & \text{if }i=2, \\ 0 & \text{otherwise}.\end{array}\right.\]

Moreover, the action of $\Sym_3$ on $\mathcal{F}(\R)$ preserves the orientation.

In particular, reducing modulo 2 gives
\[H^i(\mathcal{F}(\R),\mathbb{F}_2)=\left\{\begin{array}{cc} \mathds{1} & \text{if }i=0,3, \\ \mathbf{2}\otimes\mathbb{F}_2 & \text{if }i=1,2,\\0 & \text{otherwise}.\end{array}\right.\]
\end{cor}
\begin{proof}
Let
\[\sigma:=\sum_{w\in\Sym_3}w\]
and recall the cellular homology complex $\mathcal{K}_{\Sym_3}$ from the Theorem \ref{S3equiv}. We can directly compute
\[H_3(\mathcal{F}(\R),\Z)=\ker \partial_3=\Z\left<\sigma\right>\simeq\Z.\]
\indent We determine an orientation of $\mathcal{F}(\R)$ by choosing as fundamental class
\[[\mathcal{F}(\R)]:=\sigma.\]
Thus, for $w\in\Sym_3$ one has $[\mathcal{F}(\R)]\cdot w=[\mathcal{F}(\R)]$ and so, the right action of $\Sym_3$ on $\mathcal{F}(\R)$ preserves the orientation. Denoting by 
\[\mathcal{D}^i:=\left([\mathcal{F}(\R)]\cap -\right) : H^i(\mathcal{F}(\R),\Z) \stackrel{\tiny{\sim}}\to H_{3-i}(\mathcal{F}(\R),\Z)\]
the associated Poincar\'e duality, the naturality theorem (see \cite[Theorem 67.2]{munkreselts}) yields
\[w_*\mathcal{D}^iw^*=\mathcal{D}^i.\]
For a right $\Sym_3$-set $X$, we naturally write $X^{op}$ for the left $\Sym_3$-set $X$ endowed with the action $w\cdot x:=xw^{-1}$. Then, the last equation becomes a reformulation of the property
\[\mathcal{D}^i\in\ho_{\Z[\Sym_3]}\left(H^i(\mathcal{F}(\R),\Z),H_{3-i}(\mathcal{F}(\R),\Z)^{op}\right)\]
and the left modules $H^i(\mathcal{F}(\R),\Z)$ and $H_{3-i}(\mathcal{F}(\R),\Z)^{op}$ are thus isomorphic.

\indent We have show that $H_1(\mathcal{F}(\R),\Z)^{op}\simeq\overline{\mathbf{2}}$. Denote respectively by $x$ and $y$ the classes of $\left(\begin{smallmatrix}1+s_\beta\\0\\0\end{smallmatrix}\right)\in\ker \partial_1$ and $\left(\begin{smallmatrix}s_\alpha+s_\beta s_\alpha\\0\\0\end{smallmatrix}\right)\in\ker \partial_1$ in $H_1(\mathcal{F}(\R),\Z)$. Then we have $H_1(\mathcal{F}(\R),\Z)=\Z\left<x,y\right>\simeq\left(\Z/2\Z\right)^2$ and since
\[x+y+\left(\begin{smallmatrix}s_\alpha s_\beta+w_0\\0\\0\end{smallmatrix}\right)=\left(\begin{smallmatrix}\sigma\\0\\0\end{smallmatrix}\right)=\partial_2\left(\begin{smallmatrix}1+2 s_\alpha-s_\beta s_\alpha+s_\alpha s_\beta \\ 1+s_\alpha+s_\beta \\ -1-s_\beta-s_\beta s_\alpha\end{smallmatrix}\right)\]
we get 
\[y\cdot s_\beta=\left(\begin{smallmatrix}s_\alpha s_\beta+w_0\\0\\0\end{smallmatrix}\right)=-x-y.\] 
Next, it is easy to compute that $x\cdot s_\alpha=y$, $x\cdot s_\beta=x$ and $y\cdot s_\alpha=x$. These equations mean that, with respect to the basis $\{x,y\}$ of the free $\mathbb{F}_2$-module $H_1(\mathcal{F}(\R),\mathbb{F}_2)^{op}$, the matrices of the action of $s_\alpha$ and $s_\beta$ are given by
\[\mathrm{Mat}_{\{x,y\}}(s_\alpha)=\begin{pmatrix}0 & 1\\1 & 0\end{pmatrix},~~\mathrm{Mat}_{\{x,y\}}(s_\beta)=\begin{pmatrix}1 & 0\\1 & 1\end{pmatrix}\]
and these are indeed the matrices defining $\mathbf{2}\otimes\mathbb{F}_2$.
\end{proof}

Finally, using Figure \ref{octarep}, we can describe the $3$-cells in a more combinatorial way. More precisely, one can describe all the curved tetrahedra having a given element $w\in\Sym_3$ in its boundary. By right multiplication by $w^{-1}$, we may assume that $w=1$. First consider the octahedron as in Figure \ref{octarep}, with vertices (and centers of faces) given by the images of the ones of \ref{octarep} under the projection $\pi^\BO : \BO \twoheadrightarrow \Sym_3$ as in Figure \ref{cellwith1}. A curved tetrahedron containing $1$ can be described in the following way:
\begin{enumerate}
\item Choose a face $F$ of the octahedron,
\item Choose an edge of $F$,
\item The curved tetrahedron has its vertices given by the center of $F$, the two vertices of the chosen edge of $F$ and $1$.
\end{enumerate}
\begin{center}
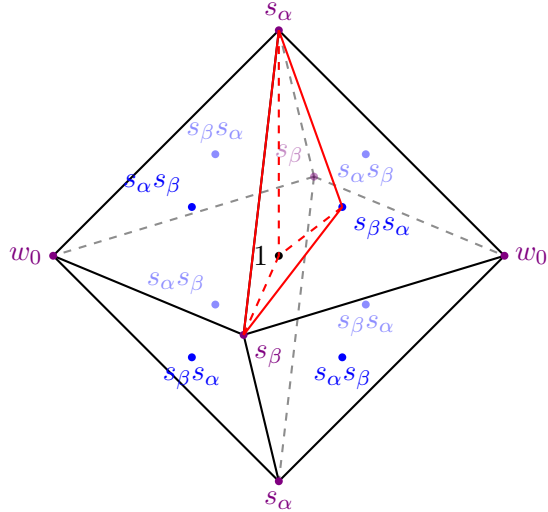
\begin{figure}[h!]
\begin{tikzpicture}[x={(1cm,0cm)},y={(0cm,1cm)},z={(-2mm,-4.5mm)},scale=3,thick]
  \coordinate (u) at (0,0,0);
  \coordinate (tip) at (0,0,0.78);
  \coordinate (tim) at (0,0,-0.78);
  \coordinate (tjp) at (1,0,0);
  \coordinate (tjm) at (-1,0,0);
  \coordinate (tkp) at (0,1,0);
  \coordinate (tkm) at (0,-1,0);
  \coordinate (o0p) at (1/3,1/3,0.26);
  \coordinate (o0m) at (-1/3,-1/3,-0.26);
  \coordinate (ojp) at (-1/3,1/3,0.26);
  \coordinate (ojm) at (1/3,-1/3,-0.26);
  \coordinate (oim) at (-1/3,-1/3,0.26);
  \coordinate (oip) at (1/3,1/3,-0.26);
  \coordinate (okp) at (1/3,-1/3,0.26);
  \coordinate (okm) at (-1/3,1/3,-0.26);
  
  \draw (tip)--(tjp);
  \draw (tip)--(tjm);
  \draw (tip)--(tkp);
  \draw (tip)--(tkm);
  \draw[opacity=0.45,dashed] (tim)--(tjp);
  \draw[opacity=0.45,dashed] (tim)--(tjm);
  \draw[opacity=0.45,dashed] (tim)--(tkp);
  \draw[opacity=0.45,dashed] (tim)--(tkm);
  \draw (tjm)--(tkm);
  \draw (tkm)--(tjp);
  \draw (tjp)--(tkp);
  \draw (tkp)--(tjm);
  
  \fill[fill=black] (u) circle (0.5pt);
  \fill[fill=black,color=violet] (tip) circle (0.5pt);
  \fill[fill=black,opacity=0.45,dashed,color=violet] (tim) circle (0.5pt);
  \fill[fill=black,color=violet] (tjp) circle (0.5pt);
  \fill[fill=black,color=violet] (tjm) circle (0.5pt);
  \fill[fill=black,color=violet] (tkp) circle (0.5pt);
  \fill[fill=black,color=violet] (tkm) circle (0.5pt);
  \fill[fill=black,color=blue] (o0p) circle (0.5pt);
  \fill[fill=black,opacity=0.45,dashed,color=blue] (o0m) circle (0.5pt);
  \fill[fill=black,opacity=0.45,dashed,color=blue] (oip) circle (0.5pt);
  \fill[fill=black,color=blue] (oim) circle (0.5pt);
  \fill[fill=black,color=blue] (ojp) circle (0.5pt);
  \fill[fill=black,opacity=0.45,dashed,color=blue] (ojm) circle (0.5pt);
  \fill[fill=black,color=blue] (okp) circle (0.5pt);
  \fill[fill=black,opacity=0.45,dashed,color=blue] (okm) circle (0.5pt);
  
  \draw (u) node[left]{$1$};
  \draw[color=violet] (tip) node[below right]{$s_\beta$};
  \draw[opacity=0.45,color=violet] (tim) node[above left]{$s_\beta$};
  \draw[color=violet] (tjp) node[right]{$w_0$};
  \draw[color=violet] (tjm) node[left]{$w_0$};
  \draw[color=violet] (tkp) node[above]{$s_\alpha$};
  \draw[color=violet] (tkm) node[below]{$s_\alpha$};
  \draw[color=blue] (o0p) node[below right]{$s_\beta s_\alpha$};
  \draw[opacity=0.45,color=blue] (o0m) node[above left]{$s_\alpha s_\beta$};
  \draw[opacity=0.45,color=blue] (oip) node[below]{$s_\alpha s_\beta$};
  \draw[color=blue] (oim) node[below]{$s_\beta s_\alpha$};
  \draw[color=blue] (ojp) node[above left]{$s_\alpha s_\beta$};
  \draw[opacity=0.45,color=blue] (ojm) node[below]{$s_\beta s_\alpha$};
  \draw[color=blue] (okp) node[below]{$s_\alpha s_\beta$};
  \draw[opacity=0.45,color=blue] (okm) node[above]{$s_\beta s_\alpha$};
  
  \draw[color=red] (tip)--(tkp);
  \draw[color=red] (tkp)--(o0p);
  \draw[color=red] (o0p)--(tip);
  \draw[color=red,dashed] (tkp)--(u);
  \draw[color=red,dashed] (o0p)--(u);
  \draw[color=red,dashed] (tip)--(u);
\end{tikzpicture}
\caption{A curved tetrahedron in $\mathcal{F}(\R)$ containing $1$ in its boundary.}
\label{cellwith1}
\end{figure}
\end{center}
\begin{rem}
Note that in this representation, many different cells can have the same vertices. For instance, the $1$-cell formed by the edge linking $1$ to the $w_0$ on the right, and then from the other copy of $w_0$ on the left, back to one is not a trivial path in $\mathcal{F}(\R)$. In fact, it corresponds to the element $j$ of the group $\mathcal{Q}_8\simeq\pi_1(\mathcal{F}(\R),1)$.
\end{rem}

\section{The icosahedral case}\label{icosahedral}
\subsection{Fundamental domain}
\hfill\\

We shall use for the binary icosahedral group $\BI$ of order 120 exactly the same method as for $\BO$. First, we are looking for a fundamental domain for $\BI$ in $\Sph^3$. To do this, we consider the orbit polytope in $\R^4$
\[\Pol:=\conv(\BI).\]
\newline
\indent This polytope has 120 vertices, 720 edges, 1200 faces and 600 facets and is known as the \emph{600-cell} (or the \emph{hexacosichoron}, or even the \emph{tetraplex}). Since $\BI$ acts freely on $\Pol_3$, there must be exactly five orbits in $\Pol_3$. Here again, we consider some elements of $\BI$, also expressed in terms of the Coxeter-Moser generators $s$ and $t$ and with $\varphi:=(1+\sqrt{5})/2$:
\[\left\{\begin{array}{ll}
\sigma_i^+:=\frac{\varphi+\varphi^{-1}i+j}{2}=t, \\[.5em]
\sigma_i^-:=\frac{\varphi+\varphi^{-1}i-j}{2}=st^{-2},\end{array}\right.~~\left\{\begin{array}{ll}
\sigma_j^+:=\frac{\varphi+\varphi^{-1}j-k}{2}=ts^{-1}t, \\[.5em]
\sigma_j^-:=\frac{\varphi-\varphi^{-1}j-k}{2}=s^{-1}t,\end{array}\right.~~\left\{\begin{array}{ll}
\sigma_k^+:=\frac{\varphi+i+\varphi^{-1}k}{2}=st^{-1}, \\[.5em]
\sigma_k^-:=\frac{\varphi+i-\varphi^{-1}k}{2}=s^{-1}t^2.\end{array}\right.\]
\newline
\indent As for $\BO$, we may find explicit representatives for the $\BI$-orbits of $\Pol_3$:
\begin{prop}\label{Ifunddom}
The following tetrahedra (in $\R^4$)
\[\Delta_1:=[1,\sigma_k^-,\sigma_k^+,\sigma_i^+],~\Delta_2:=[1,\sigma_k^-,\sigma_i^+,\sigma_j^+],~\Delta_3:=[1,\sigma_k^-,\sigma_j^+,\sigma_j^-],\]
\[\Delta_4:=[1,\sigma_k^-,\sigma_j^-,\sigma_i^-],~\Delta_5:=[1,\sigma_k^-,\sigma_i^-,\sigma_k^+]\]
form a system of representatives of $\BI$-orbits of facets of $\Pol$. Furthermore, the subset of $\Pol$ defined by
\[\mathscr{D}:=\bigcup_{i=1}^5 \Delta_i\]
is a (connected) polytopal complex and is a fundamental domain for the action of $\BI$ on $\partial\Pol$.
\end{prop}
\begin{proof}
We argue as in the proof of Proposition \ref{Ofunddom}. Let $\varphi:=(1+\sqrt{5})/2$. By invariance of $\Pol$, to verify that the following 600 inequalities 
\[\left<v,x\right>\le1,\]
with $v\in (\{\pm1\}^4\rtimes\Alt_4)\cdot U$ and
\[U:=\left\{\left(\begin{smallmatrix}4-2\varphi \\ 4-2\varphi \\ 0 \\ 0\end{smallmatrix}\right),\left(\begin{smallmatrix} 2-\varphi \\ 2-\frac{3}{\varphi} \\ 1 \\ 0\end{smallmatrix}\right),\left(\begin{smallmatrix} 2\varphi-3 \\ \frac{3}{\varphi}-1 \\ \varphi-1 \\ 0\end{smallmatrix}\right),\left(\begin{smallmatrix} 2\varphi-3 \\ 2\varphi-3 \\ 2\varphi-3 \\ 1\end{smallmatrix}\right),\left(\begin{smallmatrix} \varphi-1 \\ \varphi-1 \\ \varphi-1 \\ 2-\frac{3}{\varphi} \end{smallmatrix}\right),\left(\begin{smallmatrix} 2-\varphi \\ 2-\varphi \\ 2-\varphi \\ \frac{3}{\varphi}-1\end{smallmatrix}\right),\left(\begin{smallmatrix} 2\varphi-3 \\ 2-\varphi \\ \varphi-1 \\ 4-2\varphi \end{smallmatrix}\right)\right\},\]
are valid for $\Pol$, it is enough to check those for $v\in U$ and this is straightforward. Then, the facets are given by the equalities $\left<v,x\right>=1$ and we find their vertices:
\[\vertices(\mathscr{D})=\{1,\sigma_i^\pm,\sigma_j^\pm,\sigma_k^\pm\}\]
and since $\vertices(\mathscr{D})\cap\vertices(\mathscr{D})^{-1}=\{1\}$, the Proposition \ref{VcapVm} finishes the proof.
\end{proof}
\begin{center}
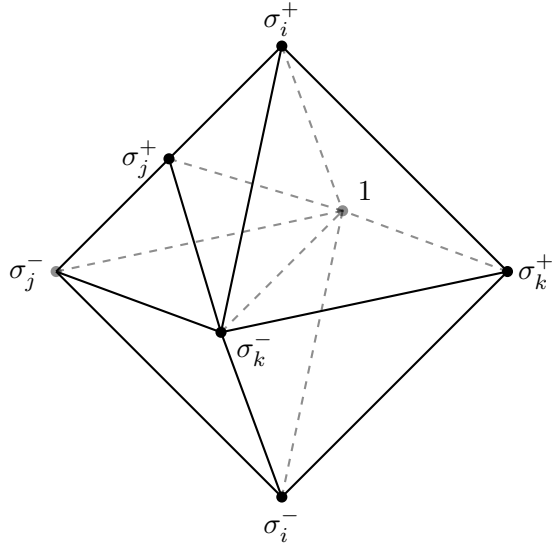
\begin{figure}[h!]
\begin{tikzpicture}[scale=3,thick]
  \coordinate (u) at (-1,0,0);
  \coordinate (s1) at (0,0,0.7);
  \coordinate (s2) at (0,0,-0.7);
  \coordinate (s3) at (0,-1,0);
  \coordinate (s4) at (-0.5,0.5,0);
  \coordinate (o1) at (1,0,0);
  \coordinate (o2) at (0,1,0);
  
  \draw (u)--(s1);
  \draw[dashed,opacity=0.45] (u)--(s2);
  \draw (u)--(s3);
  \draw (u)--(s4);
  \draw (s4)--(o2);
  \draw (s4)--(s1);
  \draw[dashed,opacity=0.45] (s4)--(s2);
  \draw[dashed,opacity=0.45] (o2)--(s2);
  \draw[dashed,opacity=0.45] (s2)--(s3);
  \draw (s3)--(o1);
  \draw (o2)--(o1);
  \draw[dashed,opacity=0.45] (o1)--(s2);
  \draw (o1)--(s1);
  \draw (s1)--(s3);
  \draw (o2)--(s1);
  \draw[dashed,opacity=0.45] (s1)--(s2);

  \fill[fill=black,opacity=0.45] (u) circle (0.7pt);
  \fill[fill=black] (s1) circle (0.7pt);
  \fill[fill=black,opacity=0.45] (s2) circle (0.7pt);
  \fill[fill=black] (s3) circle (0.7pt);
  \fill[fill=black] (s4) circle (0.7pt);
  \fill[fill=black] (o1) circle (0.7pt);
  \fill[fill=black] (o2) circle (0.7pt);
  
  \draw (0.15,0.05,0.7) node[below]{$\sigma_k^-$};
  \draw (0.1,0,-0.7) node[above]{$1$};
  \draw (s4) node[left]{$\sigma_j^+$};
  \draw (u) node[left]{$\sigma_j^-$};
  \draw (s3) node[below]{$\sigma_i^-$};
  \draw (o2) node[above]{$\sigma_i^+$};
  \draw (o1) node[right]{$\sigma_k^+$};
\end{tikzpicture}
\caption{The five tetrahedra inside $\mathscr{D}$.}
\end{figure}
\end{center}

\subsection{Associated $\BI$-cellular decomposition of $\partial\Pol$}
\hfill\\

Here also, we investigate the combinatorics of the polytopal fundamental domain $\mathscr{D}$ constructed above to obtain a cellular decomposition of it. This will give a cellular structure on $\partial\Pol$ and projecting to $\Sph^3$ gives the desired cellular structure.
\newline
\indent The facets of $\mathscr{D}$ are the ones of the five tetrahedra $\Delta_i$, except the ones that are contained in some intersection $\Delta_i\cap\Delta_j$. We obtain the following facets
\[\mathscr{D}_2=\{[1,\sigma_i^-,\sigma_k^+],[1,\sigma_k^+,\sigma_i^+],[1,\sigma_i^+,\sigma_j^+],[1,\sigma_j^+,\sigma_j^-],[1,\sigma_j^-,\sigma_i^-],\]
\[[\sigma_k^-,\sigma_i^-,\sigma_k^+],[\sigma_k^-,\sigma_k^+,\sigma_i^+],[\sigma_k^-,\sigma_i^+,\sigma_j^+],[\sigma_k^-,\sigma_j^+,\sigma_j^-],[\sigma_k^-,\sigma_j^-,\sigma_i^-]\}.\]
\newline
\indent We remark the following relations among them
\[\sigma_j^+\cdot[1,\sigma_i^-,\sigma_k^+]=[\sigma_j^+,\sigma_j^-,\sigma_k^-],~\sigma_j^-\cdot[1,\sigma_k^+,\sigma_i^+]=[\sigma_j^-,\sigma_i^-,\sigma_k^-],~\sigma_i^-\cdot[1,\sigma_i^+,\sigma_j^+]=[\sigma_i^-,\sigma_k^+,\sigma_k^-],\]
and
\[\sigma_k^+\cdot[1,\sigma_j^+,\sigma_j^-]=[\sigma_k^+,\sigma_i^+,\sigma_k^-],~\sigma_i^+\cdot[1,\sigma_j^-,\sigma_i^-]=[\sigma_i^+,\sigma_j^+,\sigma_k^-].\]
\newline
\indent These are the only relations linking facets, hence we may define the following 2-cells
\[e^2_1:=]1,\sigma_j^-,\sigma_i^-[,~e^2_2:=]1,\sigma_i^-,\sigma_k^+[,~e^2_3:=]1,\sigma_k^+,\sigma_i^+[,~e^2_4:=]1,\sigma_i^+,\sigma_j^+[,~e^2_5:=]1,\sigma_j^+,\sigma_j^-[.\]
\newline
\indent Now, define the following 1-cells
\[e^1_1:=]1,\sigma_k^+[,~~e^1_2:=]1,\sigma_i^+[,~~e^1_3:=]1,\sigma_j^+[,~~e^1_4:=]1,\sigma_j^-[,~~e^1_5:=]1,\sigma_i^-[.\]
\newline
\indent If we add to this the vertices of $\mathscr{D}$ and its interior, which is formed by only one cell $e^3$ by construction, then we may cover all of $\mathscr{D}$ with these cells and some of their translates. Thus, we have obtained the following result:
\begin{prop}\label{DPIcells}
Letting $E^0:=\{1\}$, $E^1:=\{e^1_i,~1\le i\le 5\}$, $E^2:=\{e^2_i,~1\le i \le 5\}$ and $E^3:=\{e^3\}$ with the above notations, we have the following $\BI$-equivariant cellular decomposition of the sphere
\[\Sph^3=\coprod_{\substack{0 \le j \le 3 \\ e\in E^j,g\in\BI}}g\phi(e),\]
where $p : \partial\Pol \stackrel{\tiny{\sim}}\to\Sph^3$ is the $\BI$-homeomorphism given by projection.
\end{prop}
\indent The 1-skeleton of $\mathscr{D}$ is displayed in figure \ref{1skelI}.
\begin{center}
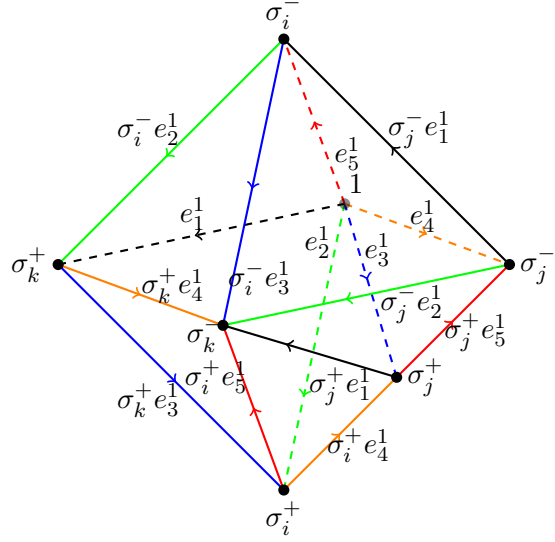
\begin{figure}[h!]
\begin{tikzpicture}[scale=3,thick]
  \coordinate (s3) at (-1,0,0);
  \coordinate (s2) at (0,0,-0.7);
  \coordinate (s1) at (0,0,0.7);
  \coordinate (u) at (0,-1,0);
  \coordinate (s4) at (0.5,-0.5,0);
  \coordinate (o1) at (0,1,0);
  \coordinate (o2) at (1,0,0);
  
  \draw[decoration={markings, mark=at position 0.5 with {\arrow{>}}},postaction={decorate},color=red] (u)--(s1);
  \draw[decoration={markings, mark=at position 0.35 with {\arrow{<}}},postaction={decorate},color=green,dashed] (u)--(s2);
  \draw[decoration={markings, mark=at position 0.5 with {\arrow{<}}},postaction={decorate},color=blue] (u)--(s3);
  \draw[decoration={markings, mark=at position 0.5 with {\arrow{>}}},postaction={decorate},color=orange] (u)--(s4);
  \draw[decoration={markings, mark=at position 0.5 with {\arrow{>}}},postaction={decorate},color=red] (s4)--(o2);
  \draw[decoration={markings, mark=at position 0.4 with {\arrow{<}}},postaction={decorate}] (s1)--(s4);
  \draw[decoration={markings, mark=at position 0.6 with {\arrow{<}}},postaction={decorate},color=blue,dashed] (s4)--(s2);
  \draw[decoration={markings, mark=at position 0.5 with {\arrow{>}}},postaction={decorate},color=orange,dashed] (s2)--(o2);
  \draw[decoration={markings, mark=at position 0.5 with {\arrow{<}}},postaction={decorate},dashed] (s3)--(s2);
  \draw[decoration={markings, mark=at position 0.5 with {\arrow{>}}},postaction={decorate},color=orange] (s3)--(s1);
  \draw[decoration={markings, mark=at position 0.5 with {\arrow{<}}},postaction={decorate}] (o1)--(o2);
  \draw[decoration={markings, mark=at position 0.5 with {\arrow{>}}},postaction={decorate},color=red,dashed] (s2)--(o1);
  \draw[decoration={markings, mark=at position 0.5 with {\arrow{<}}},postaction={decorate},color=blue] (s1)--(o1);
  \draw[decoration={markings, mark=at position 0.5 with {\arrow{<}}},postaction={decorate},color=green] (s3)--(o1);
  \draw[decoration={markings, mark=at position 0.45 with {\arrow{<}}},postaction={decorate},color=green] (s1)--(o2);

  \fill[fill=black] (u) circle (0.7pt);
  \fill[fill=black] (s1) circle (0.7pt);
  \fill[fill=black,opacity=0.45] (s2) circle (0.7pt);
  \fill[fill=black] (s3) circle (0.7pt);
  \fill[fill=black] (s4) circle (0.7pt);
  \fill[fill=black] (o1) circle (0.7pt);
  \fill[fill=black] (o2) circle (0.7pt);
  
  \draw (-0.08,-0.04,0.7) node{$\sigma_k^-$};
  \draw (0.05,0,-0.7) node[above]{$1$};
  \draw (s4) node[right]{$\sigma_j^+$};
  \draw (u) node[below]{$\sigma_i^+$};
  \draw (s3) node[left]{$\sigma_k^+$};
  \draw (o2) node[right]{$\sigma_j^-$};
  \draw (o1) node[above]{$\sigma_i^-$};
  
  \draw (0.6,0.6,0) node{$\sigma_j^-e^1_1$};
  \draw (0.33,-0.8,0) node{$\sigma_i^+e^1_4$};
  \draw (0.15,0.1,0) node{$e^1_2$};
  \draw (-0.3,-0.5,0) node{$\sigma_i^+e^1_5$};
  \draw (-0.6,-0.6,0) node{$\sigma_k^+e^1_3$};
  \draw (0.52,-0.27,0.7) node{${\sigma_j^+}e^1_1$};
  \draw (0.72,-0.02,0.4) node{$\sigma_j^- e^1_2$};
  \draw (0.30,-0.05,-0.3) node{$e^1_3$};
  \draw (0.85,-0.3,0) node{$\sigma_j^+ e^1_5$};
  \draw (0.5,0.1,-0.3) node{$e^1_4$};
  \draw (0.09,0.3,-0.5) node{$e^1_5$};
  \draw (-0.3,0.1,0.5) node{$\sigma_k^+ e^1_4$};
  \draw (-0.5,0.13,-0.25) node{$e^1_1$};
  \draw (-0.6,0.6,0) node{$\sigma_i^- e^1_2$};
  \draw (-0.05,0,0.15) node{$\sigma_i^-e^1_3$};
\end{tikzpicture}
\caption{The oriented 1-skeleton of $\mathscr{D}$.}
\label{1skelI}
\end{figure}
\end{center}
\indent We now have to compute the boundaries of the cells and the resulting cellular homology chain complex. We choose to orient the 3-cell $e^3$ undirectly, and the 2-cells directly.

\begin{center}
\begin{figure}[h!]
\begin{tikzpicture}[scale=3]
  \coordinate (s3) at (-1,0,0);
  \coordinate (s2) at (0,0,-0.7);
  \coordinate (s1) at (0,0,0.7);
  \coordinate (u) at (0,-1,0);
  \coordinate (s4) at (0.5,-0.5,0);
  \coordinate (o1) at (0,1,0);
  \coordinate (o2) at (1,0,0);
  
  \draw[thick] (u)--(s1);
  \draw[dashed,opacity=0.45] (u)--(s2);
  \draw[thick] (u)--(s3);
  \draw[thick] (u)--(s4);
  \draw[thick] (s4)--(o2);
  \draw[thick] (s1)--(s4);
  \draw[dashed,opacity=0.45] (s4)--(s2);
  \draw[dashed,opacity=0.45] (s2)--(o2);
  \draw[dashed,opacity=0.45] (s3)--(s2);
  \draw[thick] (s3)--(s1);
  \draw[thick] (o1)--(o2);
  \draw[dashed,opacity=0.45] (s2)--(o1);
  \draw[thick] (s1)--(o1);
  \draw[thick] (s3)--(o1);
  \draw[thick] (s1)--(o2);
  
  \draw[pattern=north west lines,pattern color=red,opacity=0.5] (s2)--(s3)--(u);
  \draw[pattern=north west lines,pattern color=green,opacity=0.5] (s2)--(s4)--(o2);
  \draw[pattern=north west lines,pattern color=blue,opacity=0.5] (s2)--(s4)--(u);
  \draw[pattern=north west lines,pattern color=orange,opacity=0.5] (s2)--(o1)--(o2);
  \draw[pattern=north west lines,pattern color=black,opacity=0.5] (s2)--(s3)--(o1);

  \fill[fill=black] (u) circle (0.7pt);
  \fill[fill=black] (s1) circle (0.7pt);
  \fill[fill=black,opacity=0.45] (s2) circle (0.7pt);
  \fill[fill=black] (s3) circle (0.7pt);
  \fill[fill=black] (s4) circle (0.7pt);
  \fill[fill=black] (o1) circle (0.7pt);
  \fill[fill=black] (o2) circle (0.7pt);
  
  \draw (0.45,0.35,0) node{$e^2_1$};
  \draw[decoration={markings, mark=at position 0.4 with {\arrow{>}}},postaction={decorate}] (0.45,0.35,0) ellipse (0.15 and 0.1);
  \draw (0.6,-0.16,0) node{$e^2_5$};
  \draw[decoration={markings, mark=at position 0.4 with {\arrow{>}}},postaction={decorate}] (0.6,-0.16,0) ellipse (0.13 and 0.17);
  \draw (0.3,-0.35,0) node{$e^2_4$};
  \draw[decoration={markings, mark=at position 0.4 with {\arrow{>}}},postaction={decorate}] (0.3,-0.35,0) ellipse (0.1 and 0.15);
  \draw (-0.07,-0.07,0) node{$e^2_3$};
  \draw[decoration={markings, mark=at position 0.4 with {\arrow{>}}},postaction={decorate}] (-0.07,-0.07,0) ellipse (0.13 and 0.15);
  \draw (-0.26,0.37,0) node{$e^2_2$};
  \draw[decoration={markings, mark=at position 0.3 with {\arrow{>}}},postaction={decorate}] (-0.26,0.37,0) ellipse (0.2 and 0.15);
  
\end{tikzpicture}~~~~~~\begin{tikzpicture}[scale=3,thick]
  \coordinate (s3) at (-1,0,0);
  \coordinate (s2) at (0,0,-0.7);
  \coordinate (s1) at (0,0,0.7);
  \coordinate (u) at (0,-1,0);
  \coordinate (s4) at (0.5,-0.5,0);
  \coordinate (o1) at (0,1,0);
  \coordinate (o2) at (1,0,0);
  
  \draw (u)--(s1);
  \draw (u)--(s3);
  \draw (u)--(s4);
  \draw (s4)--(o2);
  \draw (s1)--(s4);
  \draw (s3)--(s1);
  \draw (o1)--(o2);
  \draw (s1)--(o1);
  \draw (s3)--(o1);
  \draw (s1)--(o2);
  
  \fill[fill=red,opacity=0.7] (s1)--(o1)--(o2);
  \fill[fill=green,opacity=0.7] (u)--(s1)--(s3);
  \fill[fill=blue,opacity=0.7] (s1)--(s3)--(o1);
  \fill[fill=orange,opacity=0.7] (u)--(s1)--(s4);
  \fill[fill=black,opacity=0.7] (s1)--(s4)--(o2);

  \fill[fill=black] (u) circle (0.7pt);
  \fill[fill=black] (s1) circle (0.7pt);
  \fill[fill=black] (s3) circle (0.7pt);
  \fill[fill=black] (s4) circle (0.7pt);
  \fill[fill=black] (o1) circle (0.7pt);
  \fill[fill=black] (o2) circle (0.7pt);
  
  \draw (0.25,0.25,0) node{$\sigma_j^- e^2_3$};
  \draw[decoration={markings, mark=at position 0.4 with {\arrow{>}}},postaction={decorate}] (0.25,0.25,0) ellipse (0.21 and 0.2);
  \draw (0.43,-0.3,0) node{$\sigma_j^+e^2_2$};
  \draw[decoration={markings, mark=at position 0.4 with {\arrow{>}}},postaction={decorate}] (0.43,-0.3,0) ellipse (0.19 and 0.12);
  \draw (0.1,-0.6,0) node{$\sigma_i^+e^2_1$};
  \draw[decoration={markings, mark=at position 0.4 with {\arrow{>}}},postaction={decorate}] (0.1,-0.6,0) ellipse (0.2 and 0.12);
  \draw (-0.4,-0.4,0) node{$\sigma_k^+e^2_5$};
  \draw[decoration={markings, mark=at position 0.4 with {\arrow{>}}},postaction={decorate}] (-0.4,-0.4,0) ellipse (0.15 and 0.12);
  \draw (-0.5,0.15,0) node{$\sigma_i^-e^2_4$};
  \draw[decoration={markings, mark=at position 0.4 with {\arrow{>}}},postaction={decorate}] (-0.5,0.15,0) ellipse (0.2 and 0.17);
\end{tikzpicture}
\caption{The oriented 2-skeleton of $\mathscr{D}$ (back and front).}
\end{figure}
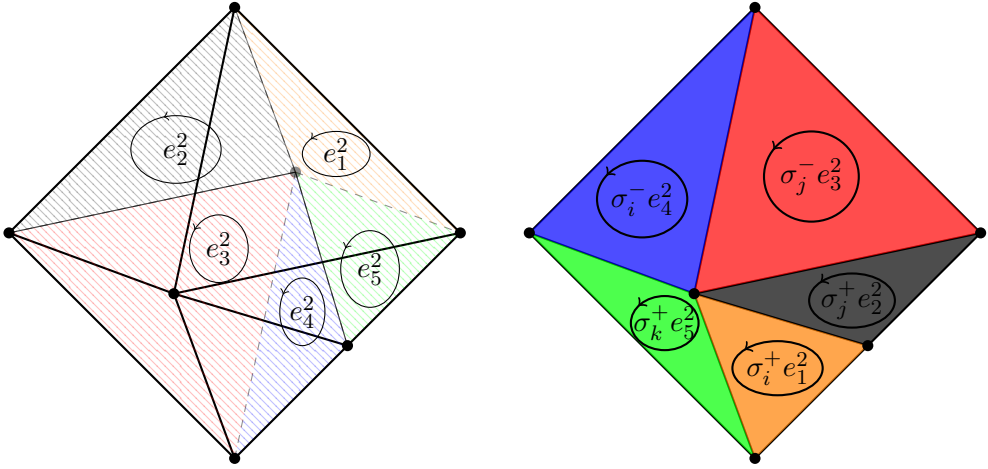
\end{center}
These orientations allow us to easily compute the boundaries of the representing cells $e^u_v$ and give the resulting chain complex of free left $\Z[\BI]$-modules
\begin{prop}\label{Ichain}
The cellular homology complex of $\partial\Pol$ associated to the cellular structure given in Proposition \ref{DPIcells} is the chain complex of free left $\Z[\BI]$-modules
\[\mathcal{K}_\BI:=\left(\xymatrix{\Z[\BI] \ar^{\partial_3}[r] & \Z[\BI]^5 \ar^{\partial_2}[r] & \Z[\BI]^5 \ar^{\partial_1}[r] & \Z[\BI]}\right),\]
where
\[\partial_1=\begin{pmatrix}\sigma_k^+-1 \\ \sigma_i^+-1 \\ \sigma_j^+-1 \\ \sigma_j^--1 \\ \sigma_i^--1\end{pmatrix},~~\partial_2=\begin{pmatrix}\sigma_j^- & 0 & 0 & 1 & -1 \\ -1 & \sigma_i^- & 0 & 0 & 1 \\ 1 & -1 & \sigma_k^+ & 0 & 0 \\ 0 & 1 & -1 & \sigma_i^+ & 0 \\ 0 & 0 & 1 & -1 & \sigma_j^+\end{pmatrix},\]
\[\partial_3=\begin{pmatrix}\sigma_i^+-1 & \sigma_j^+-1 & \sigma_j^--1 & \sigma_i^--1 & \sigma_k^+-1\end{pmatrix}.\]
\end{prop}

\subsection{The case of spheres and free resolution of the trivial $\BI$-module}
\hfill\\

Here again, we shall describe the fundamental domain obtained above in $\Sph^3$ in terms of curved join and give a fundamental domain on $\Sph^{4n-1}$ and the equivariant cellular structure on that goes with it. We finish by giving a 4-periodic free resolution of $\Z$ over $\Z[\BI]$.

\begin{theo}\label{SPHIcells}
The following subset of $\Sph^3$ is a fundamental domain for the action of $\BI$
\begin{align*}
\mathscr{F}_{3}:=&(1\ast\sigma_k^-\ast\sigma_i^+\ast\sigma_j^+) \cup (1\ast\sigma_k^-\ast\sigma_j^+\ast\sigma_j^-) \cup (1\ast\sigma_k^-\ast\sigma_j^-\ast\sigma_i^-) \\
&\cup (1\ast\sigma_k^-\ast\sigma_i^-\ast\sigma_k^+) \cup (1\ast\sigma_k^-\ast\sigma_k^+\ast\sigma_i^+).
\end{align*}

Therefore, the sphere $\Sph^3$ admits a $\BI$-equivariant cellular decomposition with the following cells as orbit representatives
\[\widetilde{e}^0:=1\ast\emptyset=\{1\},\]
\[\widetilde{e}^1_1:=\mathrm{relint}(1\ast\sigma_k^+),~\widetilde{e}^1_2:=\mathrm{relint}(1\ast\sigma_i^+),~\widetilde{e}^1_3:=\mathrm{relint}(1\ast\sigma_j^+),~\widetilde{e}^1_4:=\mathrm{relint}(1\ast\sigma_j^-),~\widetilde{e}^1_5:=\mathrm{relint}(1\ast\sigma_i^-),\]
\[\widetilde{e}^2_1:=\mathrm{relint}(1\ast\sigma_j^-\ast\sigma_i^-),~\widetilde{e}^2_2:=\mathrm{relint}(1\ast\sigma_i^-\ast\sigma_k^+),~\widetilde{e}^2_3:=\mathrm{relint}(1\ast\sigma_k^+\ast\sigma_i^+),\]
\[\widetilde{e}^2_4:=\mathrm{relint}(1\ast\sigma_i^+\ast\sigma_j^+),~\widetilde{e}^2_5:=\mathrm{relint}(1\ast\sigma_j^+\ast\sigma_j^-),~\widetilde{e}^3:=\interior{\mathscr{F}}_{3}.\]

Furthermore, the associated cellular homology complex is the chain complex $\mathcal{K}_\BI$ from the Proposition \ref{Ichain}.
\end{theo}

\begin{rem}\label{poincare}
Using the augmentation map $\varepsilon : \Z[\BI] \twoheadrightarrow \Z$, we can compute the complex $\mathcal{K}_\BI\otimes_{\Z[\BI]}\Z$ and since we have
\[\det(\partial_2\otimes\Z)=\det\left(\begin{smallmatrix}1 & 0 & 0 & 1 & -1 \\ -1 & 1 & 0 & 0 & 1 \\ 1 & -1 & 1 & 0 & 0 \\ 0 & 1 & -1 & 1 & 0 \\ 0 & 0 & 1 & -1 & 1\end{smallmatrix}\right)=1,\]
we find that $\Sph^3/\BI$ is a homology sphere, but it is not a sphere. That is, one has $H_{\ast}(\Sph^3/\BI,\Z)=H_{\ast}(\Sph^3,\Z)$, and however $\Sph^3/\BI$ is not homeomorphic to $\Sph^3$, since $\pi_1(\Sph^3/\BI)=\BI\ne1=\pi_1(\Sph^3)$.
\newline
\indent This space has a long story, it is called the \emph{Poincar\'e homology sphere}. It can also be constructed as the link of the simple singularity of type $E_8$ of the complex affine variety $\{(x,y,z)\in\C^3~;~x^2+y^3+z^5=0\}$ near the origin, as the \emph{Seifert bundle} or as the \emph{dodecahedral space}. This last one corresponds to the original construction of Poincar\'e. For a detailed expository paper on the Poincar\'e homology sphere, we refer the reader to \cite{eight_faces_poincare}.
\end{rem}

\begin{cor}\label{resolforI}
The following complex is a 4-periodic resolution of $\Z$ over $\Z[\BI]$
\[\xymatrix{\dotsc \ar[r] & \Z[\BI]^5 \ar^{\partial_{4q-3}}[r] & \Z[\BI] \ar^{\partial_{4q-4}}[r] & \dotsc\ar[r] & \Z[\BI]^5 \ar^{\partial_2}[r] & \Z[\BI]^5 \ar^<<<<<{\partial_1}[r] & \Z[\BI] \ar^<<<<<{\varepsilon}[r] & \Z \ar[r] & 0}.\]
\end{cor}

We are now able to compute the group cohomology of $\BI$ using this result. 
\begin{cor}\label{HstarI}
The group cohomology of $\BI$ with integer coefficients is given as follows:
\[\forall q\in\N,~\left\{\begin{array}{cc}
H^0(\BI,\Z)=\Z & \text{if}~~q=0, \\[.5em]
H^q(\BI,\Z)=\Z/120\Z & \text{if}~~q\equiv 0\pmod 4, \\[.5em]
H^q(\BI,\Z)=0 & \text{otherwise}. \end{array}\right.\]
\end{cor}
\begin{proof}
In view of Lemma \ref{quotient_complex}, it is suffices to compute $\mathcal{C}(\mathsf{P}^{\infty}_{\BI},\Z[\BI])\otimes_{\Z[\BI]}\Z$, with $\mathcal{C}(\mathsf{P}_{\BI}^\infty,\Z[\BI])$ the complex given in Theorem \ref{SPHNIcells}. Computing the matrices $\varepsilon(\partial_i)$ leads to the following complex
\[\xymatrix{\dotsc \ar[r] & \Z^5 \ar^0[r] & \Z \ar^{\times 120}[r] & \Z \ar[r] & \dotsc \ar[r] & \Z \ar^{\times 120}[r] & \Z \ar^0[r] & \Z^5 \ar^{\partial}[r] & \Z^5 \ar^0[r] & \Z \ar[r] & 0},\]
where $\partial=\partial_2\otimes\Z$ is the matrix given in Remark \ref{poincare}.
\end{proof}
\begin{rem}\label{TZ}
The Corollary \ref{HstarI} agrees with the previously known result on the cohomology of $\BI$, see \cite[Theorem 4.16]{tomoda-zvengrowski_2008}.
\end{rem}

\begin{theo}\label{SPHNIcells}
The chain complex $\mathcal{C}(\mathsf{P}_{\BI}^{4n-1},\Z[\BI])$ of the universal covering space of the icosahedral space forms $\mathsf{P}_{\BI}^{4n-1}$ with the fundamental group acting by covering transformations is the following complex of left $\Z[\BI]$-modules:
\[\xymatrix{0 \ar[r] & \Z[\BI] \ar^{\partial_{4n-1}}[r] & \Z[\BI]^5 \ar[r] & \dotsc\ar[r] & \Z[\BI]^5 \ar^{\partial_2}[r] & \Z[\BI]^5 \ar^{\partial_1}[r] & \Z[\BI] \ar[r] & 0},\]
where the boundaries are as in Corollary \ref{resolforI}.

In particular, the complex is exact in middle terms, i.e.
\[\forall 0<i<4n-1,~H_i(\mathcal{C}(\mathsf{P}_{\BI}^{4n-1},\Z[\BI]))=0\]
and we have
\[H_0(\mathcal{C}(\mathsf{P}_{\BI}^{4n-1},\Z[\BI]))=H_{4n-1}(\mathcal{C}(\mathsf{P}_{\BI}^{4n-1},\Z[\BI]))=\Z.\]
\end{theo}

\section{The tetrahedral case}\label{tetrahedral}
\indent Even if the case of $\BT$ has already been treated in \cite{fgnsT}, we can recover it by applying the above methods to this case. Note that all the groups in the tetrahedral family are studied in \cite{chirivi-spreafico}, but there $\BT$ is excluded since, while it is the simplest one of the family, it is somehow different from all the other ones. Since it's always the same arguments and the case is  solved, we omit the proofs.
\subsection{Fundamental domain}
\hfill\\

We consider the orbit polytope in $\R^4$
\[\Pol:=\conv(\BT).\]
\newline
\indent This polytope has 24 vertices, 96 edges, 96 faces and 24 facets and is known as the \emph{24-cells} (or the \emph{icositetrachoron}, or even the \emph{octaplex}). Since $\BT$ acts freely on $\Pol_3$, there must be exactly one orbit in $\Pol_3$. We keep the notations of the Section \ref{octahedral} and define
\[\left\{\begin{array}{llll}
\omega_i=\frac{1-i+j+k}{2}=t^{-1}s, \\[.5em]
\omega_j=\frac{1+i-j+k}{2}=st^{-1}, \\[.5em]
\omega_k=\frac{1+i+j-k}{2}=t\end{array}\right.~~~~\text{and}~~~~\left\{\begin{array}{ll}
\omega_0=\frac{1+i+j+k}{2}=s, \\[.5em]
\omega_{ij}:=\frac{1-i-j+k}{2}=t^{-1}. \end{array}\right.\]
\begin{prop}\label{Tfunddom}
The subset of $\Pol$ defined by
\[\mathscr{D}:=[1,\omega_0,\omega_j,\omega_i,\omega_{ij},k]\]
is a (connected) polytopal complex and is a fundamental domain for the action of $\BT$ on $\partial\Pol$.
\end{prop}
\begin{center}
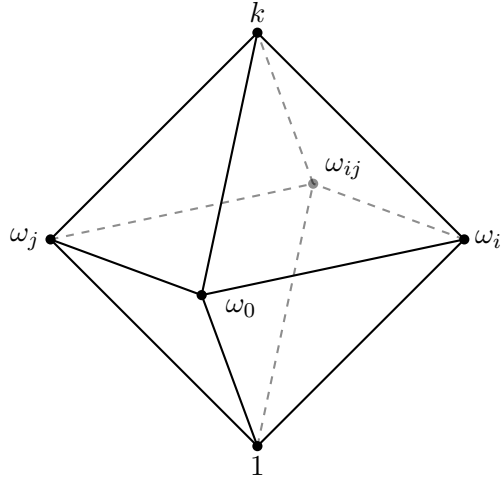
\begin{figure}[h!]
\begin{tikzpicture}[scale=2.75,thick]
  \coordinate (opm) at (-1,0,0);
  \coordinate (opp) at (0,0,0.7);
  \coordinate (omm) at (0,0,-0.7);
  \coordinate (u) at (0,-1,0);
  \coordinate (omp) at (1,0,0);
  \coordinate (k) at (0,1,0);
  
  \draw (opm)--(k);
  \draw[dashed,opacity=0.45] (opm)--(omm);
  \draw (opm)--(opp);
  \draw (opm)--(u);
  \draw (u)--(omp);
  \draw (u)--(opp);
  \draw[dashed,opacity=0.45] (u)--(omm);
  \draw[dashed,opacity=0.45] (omp)--(omm);
  \draw[dashed,opacity=0.45] (omm)--(k);
  \draw (opp)--(k);
  \draw (omp)--(k);
  \draw (opp)--(omp);
  
  \fill[fill=black] (u) circle (0.7pt);
  \fill[fill=black] (opp) circle (0.7pt);
  \fill[fill=black] (opm) circle (0.7pt);
  \fill[fill=black] (omp) circle (0.7pt);
  \fill[fill=black,opacity=0.45] (omm) circle (0.7pt);
  \fill[fill=black] (k) circle (0.7pt);
  
  \draw (0.19,0.03,0.7) node[below]{$\omega_0$};
  \draw (0.15,-0.03,-0.7) node[above]{$\omega_{ij}$};
  \draw (opm) node[left]{$\omega_j$};
  \draw (omp) node[right]{$\omega_i$};
  \draw (k) node[above]{$k$};
  \draw (u) node[below]{$1$};
\end{tikzpicture}
\caption{The tetrahedron $\mathscr{D}$.}
\end{figure}
\end{center}

\subsection{Associated $\BT$-cellular decomposition of $\partial\Pol$}
\hfill\\

The facets of $\mathscr{D}$ are the following
\[\mathscr{D}_2=\{[1,\omega_j,\omega_0],[1,\omega_0,\omega_i],[1,\omega_i,\omega_{ij}],[1,\omega_{ij},\omega_j],\]
\[[k,\omega_j,\omega_0],[k,\omega_0,\omega_i],[k,\omega_i,\omega_{ij}],[k,\omega_{ij},\omega_j]\}.\]
\newline
\indent We remark the following relations among them
\[\omega_{ij}\cdot[1,\omega_j,\omega_0]=[\omega_{ij},k,\omega_i],~\omega_j\cdot[1,\omega_0,\omega_i]=[\omega_j,k,\omega_{ij}],\]
and
\[\omega_0\cdot[1,\omega_i,\omega_{ij}]=[\omega_0,k,\omega_j],~\omega_i\cdot[1,\omega_{ij},\omega_j]=[\omega_i,k,\omega_0].\]
\newline
\indent These are the only relations linking facets, hence we may define the following 2-cells
\[e^2_1:=]1,\omega_j,\omega_0[,~e^2_2:=]1,\omega_0,\omega_i[,~e^2_3:=]1,\omega_i,\omega_{ij}[,~e^2_4:=]1,\omega_{ij},\omega_j[.\]
\newline
\indent Now, define the following 1-cells
\[e^1_1:=]1,\omega_{ij}[,~~e^1_2:=]1,\omega_j[,~~e^1_3:=]1,\omega_0[,~~e^1_4:=]1,\omega_i[.\]
\newline
\indent If we add to this the vertices of $\mathscr{D}$ and its interior, which is formed by only one cell $e^3$ by construction, then we may cover all of $\mathscr{D}$ with these cells and some of their translates. The 1-skeleton of $\mathscr{D}$ is displayed in Figure \ref{1skelT}.
\begin{center}
\begin{figure}[h!]
\begin{tikzpicture}[scale=2.75,thick]
  \coordinate (opm) at (-1,0,0);
  \coordinate (opp) at (0,0,0.7);
  \coordinate (omm) at (0,0,-0.7);
  \coordinate (u) at (0,-1,0);
  \coordinate (omp) at (1,0,0);
  \coordinate (k) at (0,1,0);
  
  \draw[decoration={markings, mark=at position 0.5 with {\arrow{>}}},postaction={decorate},color=green] (opm)--(k);
  \draw[decoration={markings, mark=at position 0.5 with {\arrow{>}}},postaction={decorate},dashed,color=blue] (opm)--(omm);
  \draw[decoration={markings, mark=at position 0.5 with {\arrow{>}}},postaction={decorate},color=orange] (opp)--(opm);
  \draw[decoration={markings, mark=at position 0.5 with {\arrow{>}}},postaction={decorate},color=red] (u)--(opm);
  \draw[decoration={markings, mark=at position 0.5 with {\arrow{>}}},postaction={decorate},color=blue] (u)--(omp);
  \draw[decoration={markings, mark=at position 0.5 with {\arrow{>}}},postaction={decorate},color=green] (u)--(opp);
  \draw[decoration={markings, mark=at position 0.5 with {\arrow{>}}},postaction={decorate},dashed,color=orange] (u)--(omm);
  \draw[decoration={markings, mark=at position 0.5 with {\arrow{>}}},postaction={decorate},dashed,color=green] (omm)--(omp);
  \draw[decoration={markings, mark=at position 0.5 with {\arrow{>}}},postaction={decorate},dashed,color=red] (omm)--(k);
  \draw[decoration={markings, mark=at position 0.5 with {\arrow{>}}},postaction={decorate},color=blue] (opp)--(k);
  \draw[decoration={markings, mark=at position 0.5 with {\arrow{>}}},postaction={decorate},color=orange] (omp)--(k);
  \draw[decoration={markings, mark=at position 0.5 with {\arrow{>}}},postaction={decorate},color=red] (omp)--(opp);
  
  \fill[fill=black] (u) circle (0.7pt);
  \fill[fill=black] (opp) circle (0.7pt);
  \fill[fill=black] (opm) circle (0.7pt);
  \fill[fill=black] (omp) circle (0.7pt);
  \fill[fill=black,opacity=0.45] (omm) circle (0.7pt);
  \fill[fill=black] (k) circle (0.7pt);
  
  \draw (0.19,0.03,0.7) node[below]{$\omega_0$};
  \draw (0.15,-0.03,-0.7) node[above]{$\omega_{ij}$};
  \draw (opm) node[left]{$\omega_j$};
  \draw (omp) node[right]{$\omega_i$};
  \draw (k) node[above]{$k$};
  \draw (u) node[below]{$1$};
  
  \draw (-0.6,-0.6,0) node{$e^1_2$};
  \draw (-0.2,-0.5,0.2) node{$e^1_3$};
  \draw (0.13,-0.6,-0.2) node{$e^1_1$};
  \draw (0.6,-0.6,0) node{$e^1_4$};
  \draw (0.5,-0.2,0.1) node{$\omega_ie^1_2$};
  \draw (0.55,0,-0.1) node{$\omega_{ij}e^1_3$};
  \draw (-0.02,0.08,-0.1) node{$\omega_je^1_4$};
  \draw (-0.45,-0.05,0.1) node{$\omega_0e^1_1$};
  \draw (0.02,0.33,-0.1) node{$\omega_{ij}e^1_2$};
  \draw (-0.3,0.45,0) node{$\omega_0e^1_4$};
  \draw (-0.62,0.62,0) node{$\omega_je^1_3$};
  \draw (0.62,0.62,0) node{$\omega_ie^1_1$};
\end{tikzpicture}
\caption{The oriented 1-skeleton of $\mathscr{D}$.}
\label{1skelT}
\end{figure}
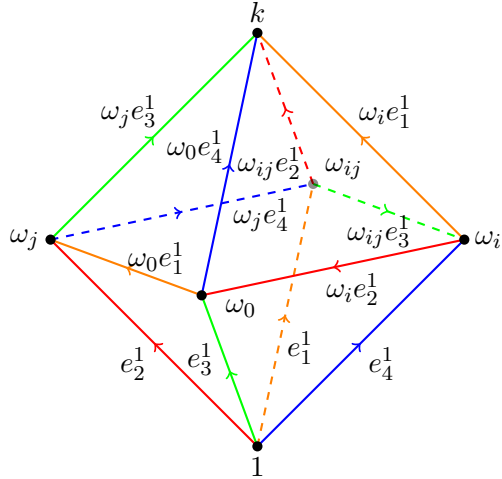
\end{center}
\begin{prop}\label{DPTcells}
Letting $E^0:=\{1\}$, $E^1:=\{e^1_i,~1\le i\le 4\}$, $E^2:=\{e^2_i,~1\le i \le 4\}$ and $E^3:=\{e^3\}$ with the above notations and denoting by $p : \partial\Pol \stackrel{\tiny{\sim}}\to \Sph^3$ the $\BT$-homeomorphism, we obtain the following $\BT$-equivariant cellular decomposition of the sphere
\[\Sph^3=\coprod_{\substack{0 \le j \le 3 \\ e\in E^j,g\in\BT}}g\phi(e).\]
\end{prop}
\indent We now have to compute the boundaries of the cells and the resulting cellular homology chain complex. We choose to orient the 3-cell $e^3$ directly, and the 2-cells undirectly.

\begin{center}
\begin{figure}[h!]
\begin{tikzpicture}[scale=2.75]
  \coordinate (s3) at (-1,0,0);
  \coordinate (s2) at (0,0,-0.7);
  \coordinate (s1) at (0,0,0.7);
  \coordinate (u) at (0,-1,0);
  \coordinate (o1) at (0,1,0);
  \coordinate (o2) at (1,0,0);
  
  \draw[thick] (u)--(s1);
  \draw[dashed,opacity=0.45] (u)--(s2);
  \draw[thick] (u)--(s3);
  \draw[thick] (u)--(o2);
  \draw[dashed,opacity=0.45] (s2)--(o2);
  \draw[dashed,opacity=0.45] (s3)--(s2);
  \draw[thick] (s3)--(s1);
  \draw[thick] (o1)--(o2);
  \draw[dashed,opacity=0.45] (s2)--(o1);
  \draw[thick] (s1)--(o1);
  \draw[thick] (s3)--(o1);
  \draw[thick] (s1)--(o2);
  
  \draw[pattern=north west lines,pattern color=orange,opacity=0.5] (s2)--(s3)--(u);
  \draw[pattern=north west lines,pattern color=red,opacity=0.5] (s2)--(o1)--(o2);
  \draw[pattern=north west lines,pattern color=green,opacity=0.5] (s2)--(s3)--(o1);
  \draw[pattern=north west lines,pattern color=blue,opacity=0.5] (s2)--(o2)--(u);

  \fill[fill=black] (u) circle (0.7pt);
  \fill[fill=black] (s1) circle (0.7pt);
  \fill[fill=black,opacity=0.45] (s2) circle (0.7pt);
  \fill[fill=black] (s3) circle (0.7pt);
  \fill[fill=black] (o1) circle (0.7pt);
  \fill[fill=black] (o2) circle (0.7pt);
  
  \draw (0.43,0.35,0) node{$\omega_{ij} e^2_1$};
  \draw[decoration={markings, mark=at position 0.4 with {\arrow{>}}},postaction={decorate}] (0.43,0.35,0) ellipse (0.20 and 0.12);
  \draw (0.4,-0.25,0) node{$e^2_3$};
  \draw[decoration={markings, mark=at position 0.5 with {\arrow{<}}},postaction={decorate}] (0.4,-0.25,0) ellipse (0.13 and 0.17);
  \draw (-0.07,-0.07,0) node{$e^2_4$};
  \draw[decoration={markings, mark=at position 0.4 with {\arrow{<}}},postaction={decorate}] (-0.07,-0.07,0) ellipse (0.13 and 0.15);
  \draw (-0.22,0.37,0) node{$\omega_je^2_2$};
  \draw[decoration={markings, mark=at position 0.3 with {\arrow{>}}},postaction={decorate}] (-0.22,0.37,0) ellipse (0.26 and 0.15);
  
\end{tikzpicture}~~~~~~\begin{tikzpicture}[scale=2.75,thick]
  \coordinate (s3) at (-1,0,0);
  \coordinate (s2) at (0,0,-0.7);
  \coordinate (s1) at (0,0,0.7);
  \coordinate (u) at (0,-1,0);
  \coordinate (o1) at (0,1,0);
  \coordinate (o2) at (1,0,0);
  
  \draw (u)--(s1);
  \draw (u)--(s3);
  \draw (u)--(o2);
  \draw (s3)--(s1);
  \draw (o1)--(o2);
  \draw (s1)--(o1);
  \draw (s3)--(o1);
  \draw (s1)--(o2);
  
  \fill[fill=orange,opacity=0.7] (s1)--(o1)--(o2);
  \fill[fill=red,opacity=0.7] (u)--(s1)--(s3);
  \fill[fill=blue,opacity=0.7] (s1)--(s3)--(o1);
  \fill[fill=green,opacity=0.7] (u)--(s1)--(o2);

  \fill[fill=black] (u) circle (0.7pt);
  \fill[fill=black] (s1) circle (0.7pt);
  \fill[fill=black] (s3) circle (0.7pt);
  \fill[fill=black] (o1) circle (0.7pt);
  \fill[fill=black] (o2) circle (0.7pt);
  
  \draw (0.25,0.25,0) node{$\omega_i e^2_4$};
  \draw[decoration={markings, mark=at position 0.4 with {\arrow{<}}},postaction={decorate}] (0.25,0.25,0) ellipse (0.2 and 0.15);
  \draw (0.2,-0.5,0) node{$e^2_2$};
  \draw[decoration={markings, mark=at position 0.4 with {\arrow{>}}},postaction={decorate}] (0.2,-0.5,0) ellipse (0.17 and 0.12);
  \draw (-0.4,-0.4,0) node{$e^2_1$};
  \draw[decoration={markings, mark=at position 0.4 with {\arrow{>}}},postaction={decorate}] (-0.4,-0.4,0) ellipse (0.1 and 0.15);
  \draw (-0.5,0.15,0) node{$\omega_0e^2_3$};
  \draw[decoration={markings, mark=at position 0.4 with {\arrow{<}}},postaction={decorate}] (-0.5,0.15,0) ellipse (0.2 and 0.17);
\end{tikzpicture}
\caption{The oriented 2-skeleton of $\mathscr{D}$.}
\label{T2cells}
\end{figure}
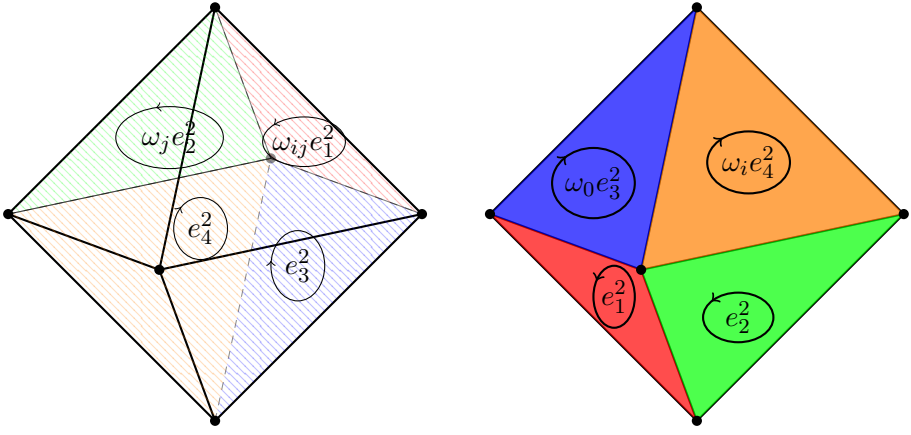
\end{center}
These orientations allow us to easily compute the boundaries of the representing cells $e^u_v$ and give the resulting chain complex of free left $\Z[\BT]$-modules
\begin{prop}\label{Tchain}
The cellular homology complex of $\partial\Pol$ associated to the cellular structure given in Proposition \ref{DPTcells} is the chain complex of free left $\Z[\BT]$-modules
\[\mathcal{K}_\BT:=\left(\xymatrix{\Z[\BT] \ar^{\partial_3}[r] & \Z[\BT]^4 \ar^{\partial_2}[r] & \Z[\BT]^4 \ar^{\partial_1}[r] & \Z[\BT]}\right),\]
where
\[\partial_1=\begin{pmatrix}\omega_{ij}-1 \\ \omega_j-1 \\ \omega_0-1 \\ \omega_i-1\end{pmatrix},~~\partial_2=\begin{pmatrix}\omega_0 & -1 & 1 & 0 \\ 0 & \omega_i & -1 & 1 \\ 1 & 0 & \omega_{ij} & -1 \\ -1 &  1 & 0 & \omega_j\end{pmatrix},\]
\[\partial_3=\begin{pmatrix}1-\omega_{ij} & 1-\omega_j & 1-\omega_0 & 1-\omega_i\end{pmatrix}.\]
\end{prop}

\subsection{The case of spheres and free resolution of the trivial $\BT$-module}
\hfill\\

Here again, we shall describe the fundamental domain obtained above in $\Sph^3$ in terms of curved join and give a fundamental domain on $\Sph^{4n-1}$ and the equivariant cellular structure on that goes with it. We finish by giving a 4-periodic free resolution of $\Z$ over $\Z[\BT]$.
\begin{theo}\label{SPHTcells}
The following subset of $\Sph^3$ is a fundamental domain for the action of $\BT$
\[\mathscr{F}_{3}:=(1\ast\omega_{ij}\ast\omega_i\ast\omega_0\ast\omega_j)\cup(\omega_{ij}\ast\omega_i\ast\omega_0\ast\omega_j\ast k).\]
In particular, the sphere $\Sph^3$ admits a $\BT$-equivariant cellular decomposition with the following cells as orbit representatives
\[\widetilde{e}^0:=1\ast\emptyset=\{1\},\]
\[\widetilde{e}^1_1:=\mathrm{relint}(1\ast\omega_{ij}),~\widetilde{e}^1_2:=\mathrm{relint}(1\ast\omega_j),~\widetilde{e}^1_3:=\mathrm{relint}(1\ast\omega_0),~\widetilde{e}^1_4:=\mathrm{relint}(1\ast\omega_i),\]
\[\widetilde{e}^2_1:=\mathrm{relint}(1\ast\omega_j\ast\omega_0),~\widetilde{e}^2_2:=\mathrm{relint}(1\ast\omega_0\ast\omega_i),~\widetilde{e}^2_3:=\mathrm{relint}(1\ast\omega_i\ast\omega_{ij}),~\widetilde{e}^2_4:=\mathrm{relint}(1\ast\omega_{ij}\ast\omega_j),\]
\[\widetilde{e}^3:=\interior{\mathscr{F}}_{3}.\]

Furthermore, the associated cellular homology complex is the chain complex $\mathcal{K}_\BT$ from the Proposition \ref{Tchain}.
\end{theo}

\begin{cor}\label{resolforT}
The following chain complex is a 4-periodic free resolution of $\Z$ over $\Z[\BT]$
\[\xymatrix{\dotsc \ar[r] & \Z[\BT]^4 \ar^{\partial_{4q-3}}[r] & \Z[\BT] \ar^{\partial_{4q-4}}[r] & \dotsc\ar[r] & \Z[\BT]^4 \ar^{\partial_2}[r] & \Z[\BT]^4 \ar^<<<<<{\partial_1}[r] & \Z[\BT] \ar^<<<<<{\varepsilon}[r] & \Z \ar[r] & 0},\]
where
\[\partial_{4q-3}=\begin{pmatrix}\omega_{ij}-1 \\ \omega_j-1 \\ \omega_0-1 \\ \omega_i-1\end{pmatrix},~~\partial_{4q-2}=\begin{pmatrix}\omega_0 & -1 & 1 & 0 \\ 0 & \omega_i & -1 & 1 \\ 1 & 0 & \omega_{ij} & -1 \\ -1 &  1 & 0 & \omega_j\end{pmatrix},\]
\[\partial_{4q-1}=\begin{pmatrix}1-\omega_{ij} & 1-\omega_j & 1-\omega_0 & 1-\omega_i\end{pmatrix},~~\partial_{4q}=\begin{pmatrix}\sum_{g\in\BT}g\end{pmatrix}.\]
\end{cor}

\begin{cor}\label{HstarT}
The group cohomology of $\BT$ with integer coefficients is given as follows:
\[\forall q\ge 1,~\left\{\begin{array}{cc}
H^q(\BT,\Z)=\Z & \text{if}~~q=0, \\[.5em]
H^q(\BT,\Z)=\Z/24\Z & \text{if}~~q\equiv 0\pmod 4, \\[.5em]
H^q(\BT,\Z)=\Z/3\Z & \text{if}~~q\equiv 2 \pmod 4, \\[.5em]
H^q(\BT,\Z)=0 & \text{otherwise}.\end{array}\right.\]
\end{cor}

\begin{theo}\label{SPHNTcells}
The chain complex $\mathcal{C}(\mathsf{P}_{\BT}^{4n-1},\Z[\BT])$ of the universal covering space of the tetrahedral space forms $\mathsf{P}_{\BT}^{4n-1}$ with the fundamental group acting by covering transformations is the following complex of left $\Z[\BT]$-modules:
\[\xymatrix{0 \ar[r] & \Z[\BT] \ar^{\partial_{4n-1}}[r] & \Z[\BT]^4 \ar[r] & \dotsc\ar[r] & \Z[\BT]^4 \ar^{\partial_2}[r] & \Z[\BT]^4 \ar^{\partial_1}[r] & \Z[\BT] \ar[r] & 0},\]
where the boundaries are as in Corollary \ref{resolforT}.

In particular, the complex is exact in middle terms, i.e.
\[\forall 0<i<4n-1,~H_i(\mathcal{C}(\mathsf{P}_{\BT}^{4n-1},\Z[\BT]))=0\]
and we have
\[H_0(\mathcal{C}(\mathsf{P}_{\BT}^{4n-1},\Z[\BT]))=H_{4n-1}(\mathcal{C}(\mathsf{P}_{\BT}^{4n-1},\Z[\BT]))=\Z.\]
\end{theo}

\subsection{Simplicial structure and minimal resolution}
\hfill\\

Since we have chosen polytopal fundamental domains for $\BT$, $\BO$ and $\BI$, it is clear that we can refine our cellular decompositions to equivariant simplicial decompositions of $\Sph^3$. We will just investigate the case of $\BT$, since the other ones can be treated in a similar way. The method is trivial: just take each one of the facets $\Delta_i$ of $\Pol$ as the 3-cells and their boundary (up to multiplication) as 2-cells. 
\newline
\indent For instance, here, take as 3-cells the following open curved joins:
\[c^3_1:=]\omega_0,1,\omega_{ij},\omega_j[,~c^3_2:=]\omega_0,1,\omega_i,\omega_{ij}[,~c^3_3:=]\omega_0,\omega_{ij},k,\omega_j[,~c^3_4:=]k,\omega_0,\omega_i,\omega_{ij}[\]
and as 2-cells the following open triangles:
\[\forall 1\le i\le 4,~c^2_i:=e^2_i\]
and
\[c^2_5:=]\omega_0,1,\omega_{ij}[,~c^2_6:=]\omega_{ij},\omega_j,\omega_0[,~c^2_7:=]\omega_0,\omega_i,\omega_{ij}[,~c^2_8:=]\omega_0,k,\omega_{ij}[\]
and we may keep the 1-cells as they are, i.e. $c^1_i:=e^1_i$ for $1\le i \le 4$. Then, the resulting simplicial homology complex is easily computed (for example, by orienting the 3-cells directly), just as we did above. One shall find of course a complex that is homotopy equivalent to the complex $\mathcal{K}_\BT$ defined in Theorem \ref{SPHTcells}. We omit the details.
\newline
\indent We conclude by discussing the minimal resolution. Group resolution and group cohomology are purely algebraic invariants of the given group $G$. Under this point of view, Swan \cite{swan_1965} proved the existence of a minimal periodic free resolution of $\Z$ over $G$, for a family of finite groups containing the spherical space form groups. This means a resolution with minimal $\Z[G]$ module's ranks. He also gave a bound for these ranks. This point has been discussed in \cite{chirivi-spreafico} for the resolution over the groups $P'_{8\cdot 3^s}$ of the tetrahedral family. Here, we show how to ``reduce'' our resolution for $\BT$ to the minimal one, that has ranks 1-2-2-1, compare \cite[10.6]{chirivi-spreafico}. (We note that in \cite[10.5]{chirivi-spreafico} there is a missprint: one should read $f_h(F^{\bullet})$ instead of $\mu_h(G)$ in the statement of the proposition.) We first describe the underlying geometric idea, and next we give an explicit chain homotopy.

Geometrically, the construction is as follows: start with the cellular decomposition from Theorem \ref{SPHTcells}. As seen in Figure \ref{T2cells}, the four upper triangles are sent by different group elements to the four lower triangles. It is clear that there is no way of collecting two triangles in one single 2-cell but we may proceed as follows. Pick up one triangle, say $e^2_1$, and one of its neighbours, say $\omega_0 e^2_3$ and set $a_1$ to be the union of these two triangles, namely
\[a_1:=e^2_1+e^2_2.\]
Then, we have that $\omega_{ij}a_1=\omega_{ij}e^2_1+\omega_{ij}e^2_2$ and $y:=\omega_{ij}e^2_2$ does not belong to the boundary of the fundamental domain $\mathscr{F}_{\BT,3}$. However, we may find an other pair of coherent triangles such that one of them is mapped to $y$ by some group element, while the other one is mapped to some triangle in the boundary of $\mathscr{F}_{\BT,3}$. For example, take
\[a_2:=\omega_0e^2_3+\omega_je^2_2.\]
Then, we have $\omega_0^{-1}a_2=e^2_3+y$. As a consequence,
\[\omega_0^{-1}a_2-\omega_{ij}a_1=e^2_3-\omega_{ij}e^2_1\]
and this means that we can use the three 2-cells $a_1$, $a_2$ and $e^2_4$ to cover all the boundary of $\mathscr{F}_{\BT,3}$. We would like to add one more triangle to the first two 2-cells in order to reduce the total number to two, but we easily see that the same procedure fails. However, we may proceed in the following ``dual'' way. Let $x$ be a triangle such that $\omega_0^{-1}x=e^2_4$ and $\omega_{ij}x=\omega_ie^2_4$. We can take $x:=]i,\omega_j,\omega_0[$ and then we define
\[b_1:=a_1+x=e^2_1+e^2_2+x\]
and
\[b_2:=a_2+x=\omega_0e^2_3+\omega_je^2_2+x.\]
Then, after a simple calculation, we find that
\[b_1-b_2+\omega_0^{-1}b_2-\omega_{ij}b_1=a_1-a_2+\omega_0^{-1}a_2-\omega_{ij}a_1+\omega_0^{-1}x-\omega_{ij}x\]
\[=(1-\omega_{ij})e^2_1+(1-\omega_j)e^2_2+(1-\omega_0)e^2_3+(1-\omega_i)e^2_4=d_3(e^3),\]
that is, the whole boundary of $\mathscr{F}_{\BT,3}$ is obtained using only the two $2$-chains $b_1$ and $b_2$.
\newline
\indent We can then give the reduced complex. It is given by the following
\[\mathcal{K}_{\BT}':=\left(\xymatrix{0 \ar[r] & K'_3 \ar^{\partial_1'}[r] & K'_2 \ar^{\partial_2'}[r] & K'_1 \ar^{\partial_1'}[r] & K'_0 \ar[r] & 0}\right),\]
where $K'_0=\Z[\BT]\left<f^0\right>$, $K'_3=\Z[\BT]\left<f^3\right>$, $K'_1=\Z[\BT]\left<f^1_1,f^1_2\right>$ and $K'_2=\Z[\BT]\left<f^2_1,f^2_2\right>$ and 
\[\left\{\begin{array}{ll}
\partial'_3(f^3)=(1-\omega_{ij})f^2_1+(1-\omega_0)f^2_2,\\
\partial'_2(f^2_1)=(\omega_0+\omega_i-1)f^1_1+(i+1)f^1_2,\\
\partial'_2(f^2_2)=(1+(-i))f^1_1+(\omega_j-1+\omega_{ij})f^1_2,\\
\partial'_1(f^1_1)=(\omega_j-1)f^0,\\
\partial'_1(f^1_2)=(\omega_i-1)f^0,\end{array}\right.\]
i.e. are given in the canonical bases by right multiplication by the following matrices
\[\partial'_1=\begin{pmatrix}\omega_j-1 \\ \omega_i-1\end{pmatrix},~~\partial'_2=\begin{pmatrix}\omega_0+\omega_i-1 & 1+i \\ 1+(-i) & \omega_j-1+\omega_{ij} \end{pmatrix},~~\partial'_3=\begin{pmatrix}1-\omega_{ij} & 1-\omega_0\end{pmatrix}.\]
\newline
\indent We finish by giving explicit homotopy equivalences $\varphi : \mathcal{K}_\BT \to \mathcal{K}_\BT'$ and $\varphi' : \mathcal{K}_\BT' \to \mathcal{K}_\BT$. We define $\varphi(e^i):=f^i$ and $\varphi'(f^i):=e^i$ for $i=0,3$ as well as 
\[\left\{\begin{array}{lll}
\varphi_2(e^2_1):=f^2_1, \\
\varphi_2(e^2_2)=\varphi_2(e^2_4):=0, \\
\varphi_2(e^2_3):=f^2_2,\end{array}\right.~~~~\text{and}~~~~\left\{\begin{array}{ll}
\varphi_2'(f^2_1):=e^2_1+e^2_2+\omega_0e^2_4, \\
\varphi_2'(f^2_2):=\omega_{ij}e^2_2+e^2_3+e^2_4 & \end{array}\right.\]
also
\[\left\{\begin{array}{lll}
\varphi_1(e^1_1):=f^1_1+\omega_jf^1_2, \\
\varphi_1(e^1_2):=f^1_1, \\
\varphi_1(e^1_3):=f^1_2+\omega_if^1_1, \\
\varphi_1(e^1_4):=f^1_2,\end{array}\right.~~~~\text{and}~~~~\left\{\begin{array}{ll}
\varphi_1'(f^1_1):=e^1_2, \\
\varphi_1'(f^1_2):=e^1_4.\end{array}\right.\]
We immediately check that $\varphi\circ\varphi'=id_{\mathcal{K}_\BT'}$ and we just have to show that the other composition is homotopic to $id_{\mathcal{K}_\BT}$. If we define $H : \mathcal{K}_\ast \to \mathcal{K}_{\ast+1}$ by $H_0=H_2=0$, $H_1(e^1_2)=H_1(e^1_4):=0$ and $H_1(e^1_1):=e^2_4$, $H_1(e^1_3):=e^2_2$, then we have $\varphi_1'\varphi_1=id+\partial_2H_1+H_0\partial_1$ and $\varphi_2'\varphi_2=id+\partial_3H_2+H_1\partial_2$, i.e.
\[\varphi'\circ\varphi=id_{\mathcal{K}_\BT}+\partial H+H\partial\]
and $\varphi$ is indeed a homotopy equivalence, with homotopy inverse $\varphi'$. Thus, we have proved that the complex $\mathcal{K}_\BT$ from the Theorem \ref{SPHTcells} is homotopy equivalent to the complex
\[\mathcal{K}_\BT'=\left(\xymatrix{0\ar[r] & \Z[\BT] \ar^{\partial'_3}[r] & \Z[\BT]^2 \ar^{\partial'_2}[r] & \Z[\BT]^2 \ar^{\partial'_1}[r] & \Z[\BT] \ar[r] & 0}\right)\]
defined above.
\begin{rem} Observe that this process works for the group $\BT$ but fails for the other two groups, $\BO$ and $\BI$. This is not unexpected, since the resolutions determined in the present work are characterised by their geometric feature, i.e. constructed through particular orthogonal representations of the groups, and it is not  likely that this characterisation would produce a minimal resolution, that in general may not be induced by a representation. Indeed, it would be interesting to investigate the possible bounds for the ranks of a free periodic resolution induced by a linear representation.
\end{rem}


\end{document}